\documentclass[11pt]{amsart}
\usepackage{amssymb,mathrsfs,subcaption,graphicx, enumerate}
\usepackage{amsmath,amsfonts,amssymb,amscd,amsthm,bbm,bm}
\usepackage{verbatim}
\usepackage[retainorgcmds]{IEEEtrantools}

\usepackage{colortbl}
\usepackage{tikz}
\usetikzlibrary{patterns}
\usetikzlibrary{positioning}
\usepackage{caption}
\usepackage{enumitem}
\usepackage{bm}

\usepackage{hyperref}
\hypersetup{
    colorlinks=true,
    linkcolor=blue,
    citecolor=magenta,
    }

\allowdisplaybreaks
\title[]{Cutoff and Dynamical Phase Transition for the General Multi-component Ising Model}

\author[Yang]{Seoyeon Yang}
\address[Seoyeon Yang]{\newline Department of Mathematical Sciences\newline Seoul National University, Seoul 08826, Korea}
\email{aromeyang@snu.ac.kr}

\usepackage{amsthm,thmtools}
\newtheorem{theorem}{Theorem}[section]
\newtheorem{lemma}[theorem]{Lemma}

\newtheorem{proposition}[theorem]{Proposition}
\newtheorem{remark}[theorem]{Remark}

\newtheorem{definition}[theorem]{Definition}

\newcommand{\E}{\mathbb{E}}
\newcommand{\prob}{\mathbb{P}}
\newcommand{\Var}{\mathrm{Var}}
\newcommand{\bS}{{\bm{S}}}
\newcommand{\bs}{{\bm{s}}}
\newcommand{\bX}{{\bm{X}}}
\newcommand{\bone}{\mathbf{1}}
\newcommand{\mo}{\chi}
\newcommand{\ba}{{\bm{a}}}
\newcommand{\bx}{{\bm{x}}}

\makeatletter
\@namedef{subjclassname@2020}{\textup{2020} Mathematics Subject Classification}
\makeatother

\begin{document}

\date{\today}

\begin{abstract}
We study the multi-component Ising model, which is also known as the block Ising model. In this model, the particles are partitioned into a fixed number of groups with a fixed proportion, and the interaction strength is determined by the group to which each particle belongs. We demonstrate that the Glauber dynamics on our model exhibits the cutoff$\mbox{--}$metastability phase transition as passing the critical inverse-temperature $\beta_{cr}$, which is determined by the proportion of the groups and their interaction strengths, regardless of the total number of particles. For $\beta<\beta_{cr}$, the dynamics shows a cutoff at $\alpha n\log n$ with a window size $O(n)$, where $\alpha$ is a constant independent of $n$. For $\beta=\beta_{cr}$, we prove that the mixing time is of order $n^{3/2}$. In particular, we deduce the so-called non-central limit theorem for the block magnetizations to validate the optimal bound at $\beta=\beta_{cr}$. For $\beta>\beta_{cr}$, we examine the metastability, which refers to the exponential mixing time. Our results, based on the position of the employed Ising model on the complete multipartite graph, generalize the results of previous versions of the model.
\end{abstract}

\maketitle \centerline{\date}

\section{Introduction}
In this paper, we study the general multi-component Curie--Weiss Ising model and demonstrate the cutoff--metastability phase transition including the rigorous power-law mixing at the critical temperature. We aim to deal with the most general case of a multi-component system in the mean-field approximation.

The phase transition between cutoff and metastability is observed universally in many cases of statistical physics, typically when the system is determined by parameters, such as temperature. The \emph{cutoff} phenomenon is an asymptotically abrupt convergence of a system to its equilibrium over negligible time referred to as the \emph{cutoff window}, which is generally considered as the convergence rate, whereas \emph{metastability} corresponds to the phenomenon when the system is trapped in local minima of energy and the exit time from these traps takes exponentially long time. There have been several efforts to associate these two phenomena although they differ in focus. Specifically, cutoff focuses on the distance between measures, whereas metastability examines exit time. After a conjecture that a necessary and sufficient condition for cutoff requires the divergence of the product, as $n$ approaches infinity, of the mixing time and spectral gap was proposed in 2004 (\cite{workshop}), Barrera, Bertoncini, and Fernández purposed to link the two phenomena in \cite{link1,link2}, and Hernández1, Kovchegov, and Otto followed suit in \cite{link3-bipartite-potts}.

Consequently, several rigorous studies have been conducted to prove cutoff--metastability transition. A popular model that exhibits this phase transition is the classical Curie--Weiss Ising model. In \cite{Curie-Ising-classic}, the total magnetization chain, which is the sum of all spins, was primarily used to prove the full cutoff--metastability phase transition with the classical coupling argument including a contraction in the mean of coupling distances. In the high temperature regime, the case of when the inverse temperature $\beta$ satisfies $\beta<1$, the dynamics exhibits a cutoff at $(2(1-\beta))^{-1}n\log n$ with a window of order $n$. At a critical temperature, the case of $\beta=1$, the mixing time is of order $n^{3/2}$. In the low temperature regime, the case of $\beta>1$, the dynamics shows metastability, wherein mixing time is exponential in $n$; moreover, the restricted system to the set of non-negative magnetization has a mixing time of order $n\log n$. The metastability is derived from the conductance arguments that a bottleneck exists between states with positive magnetization and states with negative magnetization. If the system is restricted to non-negative magnetization, the bottleneck is removed to accelerate the chain convergence. Detailed characterization and behavior of the mixing time to transition from $O(n\log n)$ to $O(n^{3/2})$ and ultimately to $O(\exp(n))$ is provided in \cite{Curie-Ising1}. In the same aspect, the phase transition of the Ising model on a regular tree was completed in \cite{Ising-tree2} by verifying polynomial mixing at criticality. Following the in-depth research of the spectral gap on $\mathbb{Z}^2$ (except for the criticality), Lubetzky and Sly demonstrated the polynomial upper bound for the critical mixing, accomplishing the full slowdown in \cite{Ising-lattice3}. Furthermore, the universal property that the Ising model in any geometry exhibits cutoff at sufficiently high temperature is proven in \cite{Ising-universal}.

There are also several studies on $q$-states Potts model, such as the Curie--Weiss Potts model and the lattice model in the case where $q$ is sufficiently large. See \cite{Curie-Potts1} and therein. In \cite{Curie-Potts1}, the critical inverse temperature $\beta_s(q)$ and the transition from cutoff to metastability passing through polynomial mixing at $\beta_s(q)$ has been proved. The authors utilized the Markov chain named proportions chain, which is the vector of the proportions of $q$ states, and a coupling argument to compute mixing times. The obtained results are similar to the Ising case because the dynamics exhibits a cutoff at $O(n\log n)$ with a window of order $n$ on $\beta<\beta_s(q)$, a polynomial mixing time at $\beta=\beta_s(q)$, and mixing time of exponential order on $\beta>\beta_s(q)$.
\vspace{0.2cm}

For the Ising model, it is conjectured that a similar cutoff--metastability phase transition phenomenon will occur in the multi-component model (Fig \ref{fig:stactic}). The multi-component Ising model has been proposed as a better approximation of lattice models than the classical Curie--Weiss Ising model. To illustrate, the model considers the partitioning of the interacting particles into several groups such that the interaction strength depends on the groups to which the particles belong (Fig \ref{fig:multi-component}). The multi-component model has received considerable attention following the research on the bipartite mean-field in \cite{bi1, bi2}, as an approximation of lattice models for meta-magnets. Moreover, as Brock proposed a new perspective discrete choice models with interactions for the Curie--Weiss model in \cite{social_eco}, the multi-component model is also considered as a representative approach to describe social phenomena or collective dynamics with social interactions among several groups. See \cite{social1, stat} and therein for more details. Thereby it has also gained immense attention from various fields besides statistical physics. Owing to increasing interest in block networking systems, multi-component models based on the Ising and Potts model has been substantially studied because of their versatility of applications in many fields including statistical physics, applied mathematics \cite{recover}, economics \cite{social_eco2}, and biology \cite{biology}.

\begin{figure}[ht] 
    \centering
    \includegraphics{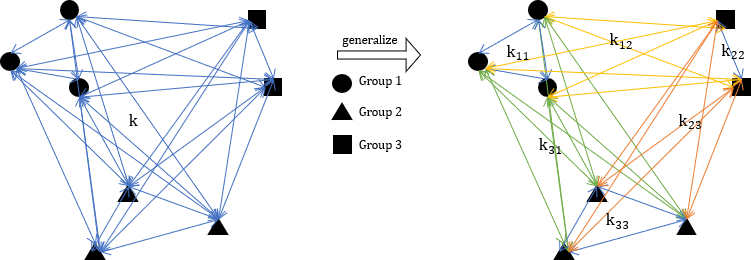}
    \caption{Generalization to a multi-component model.}
    \label{fig:multi-component}
\end{figure}

\begin{figure}
    \begin{subfigure}{0.3\linewidth}
        \centering
        \includegraphics[width=\linewidth]{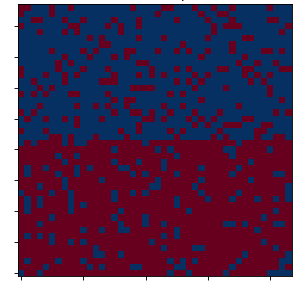}
        \caption{Low temperature}
    \end{subfigure}
    \begin{subfigure}{0.31\linewidth}
        \centering
        \includegraphics[width=\linewidth]{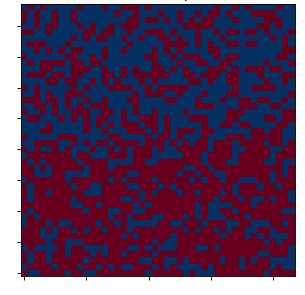}
        \caption{Critical temperature}
    \end{subfigure}
    \begin{subfigure}{0.3\linewidth}
        \centering
        \includegraphics[width=\linewidth]{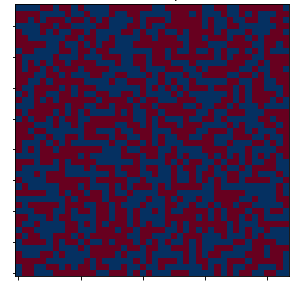}
        \caption{High temperature}
        \end{subfigure}
    \caption{Lattice snapshots at different temperatures for a 2-component model. The upper half of the plane represents group 1, while the lower half represents group 2. 
    The interaction within each group is stronger than the interaction between the two groups.
    As the temperature increases, disorder becomes more prominent.}
    \label{fig:stactic}
\end{figure}

Furthermore, there are diverse mathematical studies establishing limit theorems on the multi-component model as well as the classical Curie--Weiss Ising model. For instance, the authors in \cite{Block3} analyzed the multi-component Ising system demonstrating the large deviations principles (LDPs), central limit theorems (CLTs), and non-central limit theorems. Specifically, two methods were used to prove the CLT in the high temperature regime: the Hubbard--Stratonovich transformation, which is one of the most core techniques of showing CLT, and the multivariate version of the exchangeable pair approach in Stein’s method, which is used to determine the convergence rates in the CLT. For more studies on such limit results, see \cite{Blockeg1,Blockeg2,Blockeg4,bipartite-potts} and therein. There have also been some studies on the cutoff phenomenon in the multi-component model. To the best of our knowledge, the first attempt at a non-classical Curie--Weiss model was established in \cite{link3-bipartite-potts}; the author employed the aggregate path coupling method by extending the classical path coupling method to verify rapid mixing in the high temperature regime on the bipartite Potts model. Subsequently, \cite{Hee} proved cutoff--metastability transition of the Ising model on the complete multipartite graph, for the case of no interaction strength on the same group and homogeneous interaction strength otherwise, using the classical path coupling method.
\vspace{0.2cm}

In this paper, we obtain the critical inverse-temperature $\beta_{cr}$ explicitly and employ some techniques including an elaborate coupling argument to demonstrate the cutoff--metastability transition on the general multi-component Ising model. This is the first time full cutoff--metastability phase transition on the multi-component model with a general condition on interaction strength is proven. Moreover, based on the results obtained in this paper, our model generalizes other previous versions of the Ising model on the multipartite graph. Despite the growing interest and significant research on the multi-component model, coupled with the widespread belief that the dynamics can exhibit cutoff--metastability transition, relatively few cases have been introduced to rigorously verify this conjecture as cited in \cite{Ising-lattice3}. As a result, the full cutoff--metastability phase transition on the multi-component model has not been adequately investigated. Although the overall proof flow is a coupling argument, which is the universally accepted method, the detailed content is sufficiently delicate and sophisticated to deal with the general interaction strength.

To begin with, we verify the cutoff in the high temperature regime using the coupling argument. For an upper bound, we show various contraction properties of grand coupling, which is a generalized version of the monotone coupling in the classical Curie--Weiss Ising model, and subsequently construct a suitable coupling to establish proper supermartingale. See Section \ref{sec:Pre} for the precise definition. For a lower bound, we attain a particular starting state whose distribution of Glauber dynamics differs from the stationary measure. Despite the notable cutoff proofs in \cite{Hee, link3-bipartite-potts}, the authors assumed a homogeneous interaction strength. In contrast to this homogeneous interaction case, our generalized model is burdensome because both mean-field influence and updating probabilities are different for each group, further complicating the mean-field approximation. This challenge is alleviated by applying the mean value theorem in a novel manner along with intricate matrix computations.

Additionally, we prove power-law mixing at the critical temperature with the general settings on interaction strength. The critical case is the most demanding as observed in other models. This is because one should evince the number of steps it takes to agree in magnetizations; however, the criticality implies neither convergence nor divergence of the total-variation distance. Therefore, the ordinary coupling argument does not apply. In this paper, we observe the square sum of magnetizations and use tower property to compute its expected value to bound the mixing time at the critical temperature. Furthermore, we confirm that this bound is optimal by applying the non-central limit theorem with respect to the magnetization vector, which is also verified in this paper. In many cases, the dynamics at criticality follows a specific distribution under the stationary measure. For example, in the classical Curie--Weiss Ising model, the quantity $n^{-3/4}\sum_{v\in V}\sigma(v)$ converges weakly to a non-trivial limit law (\cite{E}). Generally, it is onerous to ascertain the limit distribution, that is, although the multi-component Ising model offers significant contributions, the results have limitations on bipartite models (\cite{bipartite-potts}) or the condition of the uniform case (\cite{Block3}). In this paper, we achieve a suitable Hubbard--Stratonovich transformation to demonstrate the non-trivial distribution of magnetizations under the stationary measure on the generalized multi-component Ising model. Consequently, we observe a specific starting configuration whose dynamics is different from the stationary one to conclude the optimal polynomial mixing time at the critical temperature.

Lastly, we prove metastability in the low temperature regime. We establish two local minima of free energy by certifying the existence of two solutions of the mean-field equation. By conductance arguments, as shown in \cite{Curie-Potts1} or \cite[Chapter 7]{MCMC}, more than one local minima of free energy function in terms of the magnetization indicate that the mixing time for the Glauber dynamics is exponential in $n$. Two metastability states are derived from the characteristic of the hyperbolic tangent function and the intermediate value theorem. This result coincides with other studies regarding the number of the ground states, such as \cite{recover1}.
\vspace{0.2cm}

In conclusion, we complete the full cutoff--metastability transition on the general multi-component Ising model. We explicitly characterize a critical inverse-temperature $\beta_{cr}$ and a constant $\alpha$ in \eqref{const} that only depend on the proportions of the group size and the interaction strength to describe the phase transition as follows.
\begin{enumerate}[label=(\roman*)] 
    \item When $\beta < \beta_{cr}$, there is a cutoff at $\alpha n\log n$ with a window of order $n$.
    \item When $\beta = \beta_{cr}$, the mixing time of order $n^{3/2}$ which is the optimal bound.
    \item When $\beta > \beta_{cr}$, the dynamics exhibits metastability and the mixing time is exponential in $n$.
\end{enumerate}

The remainder of this paper is organized as follows. In Section \ref{sec:TheModel}, we precisely define our model and introduce some general notations for the multi-component Ising model. In Section \ref{sec:MainResults}, we introduce our main results. Section \ref{sec:Pre} consists of the various properties of our model that are required to demonstrate the phase transition, such as contraction property or variance bounds. Section \ref{sec:high} is devoted to the proof of the cutoff phenomenon in the high temperature regime. In Section \ref{sec:critical}, we show the polynomial mixing time at the critical temperature. Subsequently, by proving the non-CLT property at the critical temperature, we validate that this bound is optimal. In Section \ref{sec:low}, we verify exponentially slow mixing time in the low temperature regime. 
\newline

\section{Proposed Model}\label{sec:TheModel}
In this section, we describe the multi-component Ising model. 
\subsection{Ising model and the Glauber dynamics}
Let $G=(V,E)$ be a graph with finite ($|V|=n$) vertices. The state space is defined by $\Omega=\{-1,+1\}^V$, and each element $\sigma\in\Omega$ is referred to as \emph{configuration}, i.e., each configuration is equivalent to an assignment of positive or negative spins on the sites in $V$. The nearest-neighbor \emph{Ising model} on $G$ with the inverse temperature $\beta\geq 0$ is defined on $\Omega$ by the Hamiltonian function
\[
H(\sigma)=-\sum_{v\sim w}K(v,w)\sigma(v)\sigma(w),
\]
where $x\sim y$ indicates that they are neighboring sites, and its corresponding Gibbs measure
\begin{align}\label{eq:Gibbs}
\mu(\sigma)=Z_{\beta}^{-1}\exp(-\beta H(\sigma)),
\end{align}
where $Z_\beta$ is a normalizing constant. We use a notation $\mu_n$ for the Gibbs measure to indicate $|V|=n$. The parameter $K(v,w)$ represents the interaction intensity between vertices $v$ and $w$. In this paper, we consider the ferromagnetic interaction that every $K(v,w)$ is positive. Notably, the higher the temperature rises, the weaker the influence of spin interaction is.

Throughout this paper, we employ the \emph{Glauber dynamics}, discrete-time Markov chains on the state space $\Omega$, with respect to the Gibbs distribution and denote this by $(\sigma_t)_{t\geq 0}$. It is equivalent to updating the spins by the following procedure: At each step, choose a site $v$ uniformly at random on $V$ and create a new configuration according to $\mu$ on the feasible set
\[
\Omega^v:=\{\sigma'\in\Omega : \sigma(w)=\sigma'(w) \text{ for every }w\neq v\},
\]
which implies that the probability of updating $v$ to $\pm 1$ is $r_{\pm}(S^v)$, where
\begin{align*}
r_\pm (s) := \frac{e^{\pm\beta s}}{e^{\beta s}+e^{-\beta s}}=\frac{1\pm\tanh{\beta s}}{2},
\end{align*}
and $S^v:=\sum_{w: w\sim v} K(v,w){\sigma (w)}$. The variable $S^v$ is referred to as the mean-field at site $v$. We use the notation of $\E_\sigma$ and $\prob_\sigma$ to indicate the expectation operator and probability measure with the starting state $\sigma_0:=\sigma$. Note that the Gibbs measure is the stationary measure of the Glauber dynamics.

\subsection{Markov chain mixing time and cutoff}
Our main concern is the distance between the distribution of the Markov chain by the Glauber dynamics and stationary distribution at time $t$. For analysis, we use the \emph{total-variance distance}
\[
d_n(t):=\max_{\sigma\in\Omega}\|\prob_\sigma(\sigma_t\in\cdot)-\mu_n\|_{\textrm{TV}}.
\]
The total-variation \emph{mixing time} is the number of steps until the total-variance distance is at most $\epsilon$, precisely written as
\begin{align*}
    t_{\text{mix}}^{(n)}(\epsilon):=\min\{t:d_n(t)\leq\epsilon\}.
\end{align*}
We work with $t_{\text{mix}}^{(n)}:=t_{\text{mix}}^{(n)}(1/4)$. In Section \ref{sec:high}, we observe a certain phenomenon known as \emph{cutoff}. The sequence of the Markov chain is said to exhibit \emph{cutoff} if for all $\epsilon\in(0,1)$,
\[
\lim_{n\rightarrow\infty}\frac{t_{\text{mix}}^{(n)}(\epsilon)}{t_{\text{mix}}^{(n)}(1-\epsilon)}=1.
\]
This equation implies that the total-variance distance falls drastically at the mixing time $t_{\text{mix}}^{(n)}$. In addition, to characterize the time scale the distance drops from near one to near zero, we say that a Markov chain exhibits a cutoff phenomenon with a \emph{window} of size $O(w_n)$ when $w_n=o\left(t_{\text{mix}}^{(n)}\right)$ and
\begin{align*}
    \lim_{\gamma\rightarrow\infty}\liminf_{n\rightarrow\infty} d_n(t_\text{mix}^{(n)}-\gamma w_n)&=1,\\
    \lim_{\gamma\rightarrow\infty}\limsup_{n\rightarrow\infty} d_n(t_\text{mix}^{(n)}+\gamma w_n)&=0.
\end{align*}

\subsection{Model}
We consider $G$ to be the complete graph with $n$ vertices: the case of when the interaction intensity is homogeneous such that $K(v,w)=1/n$ corresponds to the classical Curie--Weiss Ising model. In this paper, we generalize the classical Curie--Weiss Ising model to the multi-component Ising model, that is, the sites are partitioned into a finite number of groups and their interaction strength depends on the groups to which the sites belong. For some $m\in\mathbb{N}$, denote the $m$ number of groups by $G_1,G_2,\cdots,G_m$ with fixed proportion $p_1, p_2, \cdots , p_m$ to ensure that $|G_i|=np_i$ for each $1\leq i\leq m$, and the interaction is given by 
\begin{align*}
    K(v,w)=K_{ij}=k_{ij}/n,\quad\text{if }v\in G_i,~ w\in G_j.
\end{align*}
We would also call a group a \emph{block}. We consider the ferromagnetic case that we may assume $k_{ij}>0$. The interaction matrix is given by 
\[
\mathbf{J}_n:= \frac{1}{n}\begin{pmatrix}
k_{11}J_{np_1,np_1}  & k_{12}J_{np_1,np_2} & \cdots & k_{1m}J_{np_1,np_m}\\
k_{21}J_{np_2,np_1}  & k_{22}J_{np_2,np_2} & \cdots & k_{2m}J_{np_2,np_m}\\
\vdots  & \cdots& \cdots & \vdots\\
k_{m1}J_{np_m,np_1}  & k_{m2}J_{np_m,np_2} & \cdots & k_{mm}J_{np_m,np_m}
\end{pmatrix},
\]
where $J_{a,b}$ is the $a\times b$ matrix such that every element equals one. We use the notation on matrices
\begin{align}\label{eq:def_matrix}
\begin{split}
    \mathbf{K}&:=(k_{ij})_{1\leq i,j\leq m},\\
    \mathbf{P}&:=\text{diag}(p_1,p_2,\cdots,p_m),\\
    \mathbf{D}&:=\text{diag}(\sqrt{p_1},\cdots,\sqrt{p_m}),\\
    \mathbf{B}&:=(p_ik_{ij})_{1\leq i,j\leq m}=\mathbf{D}^2\mathbf{K}.
\end{split}
\end{align}
Henceforth, we would refer to the matrix $\mathbf{K}$ as the interaction matrix. We often employ the block \emph{magnetization} that are defined by the sum of the spins in each group, precisely,
\begin{align}\label{eq:def_mag}
\begin{split}
    & M^{(i)}(\sigma):=\sum_{v\in G_i} \sigma(v),~~ i=1,\cdots m,\\
    & \bm{M}(\sigma):=(M^{(1)}(\sigma),M^{(2)}(\sigma),\cdots,M^{(m)}(\sigma)),\quad \bm{S}(\sigma):=\frac{1}{n}\bm{M}(\sigma),
\end{split}
\end{align}
which are Markov chains (proven in Section \ref{sec:Pre}), and it will be often utilized instead of the configuration chain. This magnetization chain would play a key role in analyzing the mixing time. Moreover, observe that the Gibbs distribution is represented with the magnetization by
\begin{align}\label{eq:Gibbs2}
    \mu(\sigma)&=Z_{\beta}^{-1}\exp\left(\beta\sum_{v\sim w}K(v,w)\sigma(v)\sigma(w)\right)
    ={Z'}_{\beta}^{-1}\exp\left(\frac{\beta}{2n}\langle \bm{M}(\sigma),\mathbf{K} \bm{M}(\sigma)\rangle\right),
\end{align}
where $\langle\cdot,\cdot\rangle$ indicates inner product between two vectors and $Z'_\beta$ is a new normalizing constant, since 
\begin{align*}
    \frac{1}{n}\langle \bm{M}(\sigma),\mathbf{K} \bm{M}(\sigma)\rangle=\sum_{i,j=1}^m \frac{k_{ij}}{n}M^{(i)}(\sigma)M^{(j)}(\sigma)
    &=\sum_{v\sim w}K(v,w)\sigma(v)\sigma(w)+\sum_{v\in V}{K(v,v)\sigma(v)^2}\\
    &=\sum_{v\sim w}K(v,w)\sigma(v)\sigma(w)+\sum_{i=1}^m{k_{ii}p_i},
\end{align*}
of which second term is constant.

\section{Main Results}\label{sec:MainResults}
\subsection{Cutoff in the high temperature regime}
Our first result shows that the multi-component Ising model exhibits cutoff in the high temperature regime. We remark that the absence of interior interaction strength, i.e., $k_{ii}=0$ is allowed throughout this section. We first introduce the critical inverse-temperature $\beta_{cr}$, which only depends on the proportions of blocks and interaction strength regardless of the total numbers of sites $n$. We demonstrate that the Glauber dynamics for the multi-component Ising model exhibits cutoff at $\alpha n\log n$ with a window of size $O(n)$ in the high temperature regime $\beta<\beta_{cr}$, where $\alpha$ is a constant independent of $n$ determined by the inverse temperature $\beta$ and $\beta_{cr}$. Both $\beta_{cr}$ and $\alpha$ are characterized explicitly in \eqref{const} in terms of $k_{ij}$, $p_i$, and $\beta$. 
\begin{theorem}\label{Thm:HighTemp}
The Glauber dynamics for the multi-component Ising model shows a cutoff phenomenon at $t_n:=\alpha n\log n$ with window size $n$ in the high temperature regime $\beta<\beta_{cr}$, where $\alpha=\alpha(\beta,p_i,k_{ij})$ and $\beta_{cr}=\beta_{cr}(p_i,k_{ij})$ are constants defined in \eqref{const}.
\end{theorem}

Recently, there have been considerable results related to the multi-component Ising model, including various limit theorems. For example, it is known that the sequence
\[
\tilde{\bm{M}}_n=\left( \frac{1}{|G_i|}\sum_{v\in G_i}\sigma(v) \right)_{i=1,\cdots,m}
\]
satisfies an LDP under the Gibbs measure. See \cite{Block3} for a detailed description. 
The high temperature regime and low temperature regime in \cite{Block3} coincide with ours; hence, it can be applied to our model. In addition, \cite[Theorem 1.3]{Block3} implies the CLT in the high temperature regime that
\[
\left(\frac{1}{\sqrt{np_i}}\sum_{v\in G_i}\sigma(v)\right)_{i=1,\cdots,m} 
\]
converges to the normal distribution if $\mathbf{K}$ is a positive definite matrix. Moreover, according to \cite[Theorem 4]{Block3}, one can establish convergence rate in the CLT using the exchangeable pair approach of Stein’s method.

\subsection{Power-law at the critical temperature}\label{subsec:result_critical}
At the critical temperature, we not only prove that $O(n^{3/2})$ is bound on mixing time, but also that it is the optimal bound. {In this section, we assume that the interaction matrix $\mathbf{K}$ is positive definite to warrant the optimality. Generally, this condition ensures that the interaction within the groups dominates the interaction among the groups.}
\begin{theorem}\label{Thm:CriticalTemp}
When $\beta=\beta_{cr}$, there exist constants $C_1$ and $C_2$ such that the mixing time of Glauber dynamics for the multi-component Ising model satisfies
\[
C_1 n^{3/2}\leq t_{mix}(n)\leq C_2 n^{3/2}.
\]
\end{theorem}
We employ a proper coupling argument to demonstrate the upper bound. To verify the lower bound, we prove a limit theorem stated in Proposition \ref{3.3} based on a suitable Hubbard--Stratonovich transform. As described in Section \ref{sec:critical}, at the critical temperature $\beta=\beta_{cr}$, the largest eigenvalue of a positive symmetric matrix $\beta \mathbf{D}\mathbf{K}\mathbf{D}$ is one. Thus, the spectral theorem implies that it can be decomposed by
\[
\beta \mathbf{D}\mathbf{K}\mathbf{D}=\mathbf{V}\text{diag}(\lambda_1,\cdots,\lambda_{m-1},1)\mathbf{V}^\top,
\]
with the eigenvalues $0<\lambda_1\leq\lambda_2\leq\cdots\leq\lambda_{m-1}<\lambda_m =1$ and an orthogonal matrix $\mathbf{V}$ (see Section \ref{sec:critical} for details). Define matrices
\begin{align}\label{eq:def_matrix2}
\begin{split}
    \mathbf{\Gamma}_n&:=\text{diag}(n^{-1/2},n^{-1/2},\cdots,n^{-1/2},n^{-3/4}),\\
     \mathbf{C}_n &:= \text{diag}(\lambda_1,~\lambda_2,~\cdots,\lambda_{m-2},~\lambda_{m-1},~n^{1/2}),\\
     \bm{U}_n(\sigma)&:= \mathbf{\Gamma}_n \mathbf{V}^\top\mathbf{D}^{-1}\bm{M}(\sigma).
\end{split}
\end{align}
Then, we obtain the following result with regard to $\bm{U}_n$ under the Gibbs measure $\mu_n$.
\begin{proposition}\label{3.3}
Let $\bm{Y}_n\sim\mathcal{N}(0,\mathbf{C}_n^{-1})$. When $\beta=\beta_{cr}$, $\mu_n^{-1}(\bm{U}_n)+\bm{Y}_n$ converges in  distribution to a probability measure with density
\begin{align*}
    f(\bx):=\frac{1}{\tilde{Z}}\exp{\left(-\frac{1}{2}\displaystyle\sum_{i=1}^{m-1}{\Big((\lambda_i - \lambda_i^2) x_i^2\Big)}
-\frac{x_m^4}{12}\displaystyle\sum_{i=1}^{m}\frac{(V_{im})^4}{p_i}\right)}d\bx,
\end{align*}
where $\tilde{Z}$ is a normalizing constant.
\end{proposition}
\begin{remark}
By the previous proposition, the $(m-1)$-dimensional vector $\Big({U}_n^{(j)}\Big)_{j=1,\cdots,m-1}$ converges to a normal distribution with covariance matrix
\begin{align}\label{eq:Sigma}
\mathbf{\Sigma}:=\text{diag}\left((1-\lambda_1)^{-1},\cdots,(1-\lambda_{m-1})^{-1}\right),    
\end{align}
and the random variable ${U}_n^{(m)}$ converges to a distribution with Lebesgue-density
\[
\frac{1}{Z}\exp\left(-\left(\frac{1}{12}\sum_{i=1}^{m}(V_{im})^4/p_i\right)x^4\right)dx,
\]
where $Z$ is a normalizing constant, under the stationary measure $\mu_n$.
\end{remark}
In the classical Curie--Weiss model, there are substantial results concerning the limit theorems, such as CLT as shown in \cite{E} and \cite{GS}. In particular, it is well-known that $n^{-3/4}\sum_{v\in V}\sigma(v)$ converges weakly to a non-trivial measure at the critical temperature of $\beta=1$ (see \cite{Curie-Ising-classic}). The multi-component version of this fact was treated in \cite{Block3} and referred to as non-CLT in the restricted condition of the uniform case. Above Proposition \ref{3.3} proposes the generalized version of non-CLT with respect to the magnetization vector.

\subsection{Metastability in the low temperature regime}
Finally, we prove the exponential mixing time in the low temperature regime. The entropy of the system is maximum when half of the spins is $+1$ and the other half is $-1$. For a sufficiently high temperature, the entropy would dominate the energy barrier, and the average magnetization would be zero. In contrast to this situation, in the low temperature regime, it is difficult to overcome the energy barrier, which mathematically speaking denotes the existence of at least two local minima of free energy. We show the existence of two local minima to conclude exponentially long mixing time.
\begin{theorem}\label{Thm:LowTemp}
When $\beta>\beta_{cr}$, the Glauber dynamics for the multi-component Ising model has an exponential mixing time.
\end{theorem}
We remark that this result also coincides with \cite[Proposition 2.1]{recover1}, which characterized the ground states of the multi-component Ising model with two groups of equal size.

\section{Preliminaries}\label{sec:Pre}
In this section, we provide the prerequisite concepts and properties for our multi-component Ising model to achieve the main results. Before elaborating on the preliminaries, we explain some of the notations used in the paper as follows:
\renewcommand\labelitemi{$\vcenter{\hbox{\tiny$\bullet$}}$}
\begin{itemize}[leftmargin=*]
\item We write $f(x)=O(g(x))$ if there exists a positive real number $M$ and a real number $x_0$ such that $|f(x)|\leq Mg(x)$ for all $x\geq x_{0}$.
\item For any square matrix $\mathbf{A}$, $\mathbf{A}^t$ indicates the power of $t$, not the transpose of the matrix. We use $\mathbf{A}^\top$ for the transposed matrix.
\item The symbol $\circ$ means the Hadamard product, namely, the element-wise product. Precisely, $\mathbf{A}\circ\mathbf{B}=(A_{ij}B_{ij})_{i,j}$ for matrices $\mathbf{A}$ and $\mathbf{B}$ with the same dimension.
\item Given two $k-$dimensional vectors $\bm{a}$ and $\bm{b}$, we write $\bm{a}\leq \bm{b}$ when the inequality holds element-wisely, i.e., $a_i\leq b_i$ for all $1\leq i\leq k$.
\end{itemize}

Note that we consider the multi-component Ising model on $n$ sites, when the sites are partitioned into $m$ groups, or blocks, $G_1,G_2,\cdots,G_m$ with the proportion $p_1,p_2,\cdots,p_m$, i.e., $|G_i|=np_i$ for $i=1,2,\cdots,m$. The set of configurations $\Omega=\{-1,+1\}^V$ is the state space and $\mu_n$ is the Gibbs measure as given in \eqref{eq:Gibbs}. Additionally, $\E_\sigma$ and $\prob_\sigma$ indicate the expectation and probability measure of the Glauber dynamics whose starting state is $\sigma$.

\vspace{0.2cm}
In Subsection \ref{subsec:monotone coupling}, we describe a grand coupling known as the \emph{monotone coupling} and its contraction properties. We interpret such properties with the magnetization chains in Subsection \ref{subsec:magnetization chains}. Subsection \ref{subsec:variance bound} is devoted to analyzing variance bounds of the magnetization that are important to bound the order of mixing times. In Subsection \ref{subsec:after}, we prove that once the magnetizations agree, $O(n\log n)$ steps are sufficient to obtain full agreement at all temperatures, which will be primarily used at the critical temperature. Lastly, in Subsection \ref{subsec:superMG}, we state a vital lemma about supermartingales that is applied in deriving the upper bound of mixing times.

\subsection{Monotone coupling}\label{subsec:monotone coupling}
We illustrate the $\emph{monotone coupling}$ to construct Markov chains evolving together. Let $I$ be a uniform random variable over the sites $V=\{1,2,\cdots,n\}$, and generate a uniform random variable $U$ on $[0,1]$ independent of $I$. Suppose $I\in G_j$. For each $\sigma\in\Omega$, determine the spin $S^{\sigma}$ according to $I$ and $U$ by
\[
S^{\sigma}=
\begin{cases}
+1,& 0<U\leq r_+\left(\frac{\sum_{1\leq l\leq m}{k_{jl}M_l}-k_{jj}\sigma(I)}{n}\right),\\
-1,& r_+\left(\frac{\sum_{1\leq l\leq m}{k_{jl}M_l}-k_{jj}\sigma(I)}{n}\right)<U\leq 1,
\end{cases}
\]
and generate the next configuration as follows.
\[
\sigma_1(v)=
\begin{cases}
\sigma(v),& v\neq I,\\
S^{\sigma},& v=I.
\end{cases}
\]
For a given pair of configurations $\sigma, \sigma'$, we name the two-dimensional grand coupling $(\sigma_t,\sigma'_t)_{t\geq 0}$ $\emph{monotone coupling}$ with starting states $\sigma$ and $\sigma'$. We write $\sigma\leq \sigma'$ to designate the case of $\sigma(v)\leq \sigma'(v)$ for all sites $v\in G$.
\begin{remark}
If the starting states satisfy $\sigma\leq \sigma'$, then the monotone coupling maintains the monotonicity $\sigma_t\leq \sigma_t'$ for every $t\geq 0$ by means of the definition.
\end{remark}
We measure the distance between two configurations $\sigma$ and $\sigma'$ via \emph{Hamming distance}. The $i$th-$\emph{Hamming distance}$ equals to the number of sites where the two configurations disagree on the set $G_i$, i.e.,
\begin{align}\label{eq:distance}
\textnormal{dist}(\sigma^{(i)},\sigma'^{(i)}):=\frac{1}{2}\sum_{v\in G_i}|\sigma(v)-\sigma'(v)|.
\end{align}
The monotone coupling shows a contraction of the mean Hamming distance as follows.
\begin{lemma}\label{L_contr}
For the monotone coupling $(\sigma_t,\sigma'_t)$ starting at $(\sigma_0,\sigma_0')$, we have
\[
\begin{pmatrix}
\mathbb{E}\textnormal{dist}(\sigma_t^{(1)},\sigma_t'^{(1)}) \\
\mathbb{E}\textnormal{dist}(\sigma_t^{(2)},\sigma_t'^{(2)}) \\
\vdots \\
\mathbb{E}\textnormal{dist}(\sigma_t^{(m)},\sigma_t'^{(m)}) \\
\end{pmatrix}
\leq \mathbf{A}^t
\begin{pmatrix}
\textnormal{dist}(\sigma_0^{(1)},\sigma_0'^{(1)}) \\
\textnormal{dist}(\sigma_0^{(2)},\sigma_0'^{(2)}) \\
\vdots \\
\textnormal{dist}(\sigma_0^{(m)},\sigma_0'^{(m)}) \\
\end{pmatrix},
\]
where the matrix $\mathbf{A}$ is defined by
\[
\mathbf{A}= \begin{pmatrix}
1-\frac{1}{n}+\beta K_{11}(p_1-\frac{1}{n})  & \beta p_1 K_{12} & \cdots & \beta p_1 K_{1m}\\
\beta p_2 K_{21} & 1-\frac{1}{n}+\beta K_{22}(p_2-\frac{1}{n})  & \cdots & \beta p_2 K_{2m}\\
\vdots  & \cdots& \cdots & \vdots\\
\beta p_m K_{m1} & \beta p_m K_{m2} & \cdots & 1-\frac{1}{n}+\beta K_{mm}(p_m-\frac{1}{n})
\end{pmatrix}.
\]
\end{lemma}
With simpler notation, we can write
\begin{align*}
(\mathbf{A}_{ij})=\left.
  \begin{cases}
    1-\frac{1}{n}+\beta K_{ii}\left(p_i-\frac{1}{n}\right), & \text{for } i=j \\
    \beta p_i K_{ij}, & \text{for } i \neq j
  \end{cases}
  \right\},
\end{align*}
and therefore
\[
\mathbf{A}=(1-\frac{1}{n})\mathbf{I}_n + \frac{\beta}{n}\mathbf{B}-\frac{\beta}{n^2}\mathbf{C},
\]
where $(\mathbf{B})_{ij}=p_i k_{ij}$ and $\mathbf{C}:=\text{diag}\left(k_{11},\cdots,k_{mm}\right)$. Since $\mathbf{A}\leq (1-\frac{1}{n})\mathbf{I}_n + \frac{\beta}{n}\mathbf{B}$, we will often employ 
\begin{align}\label{eq:def_Q}
    \mathbf{Q}:=(1-\frac{1}{n})\mathbf{I}_n + \frac{\beta}{n}\mathbf{B}
\end{align}
instead of $\mathbf{A}$ in the aforementioned equation.
\begin{proof}
We first show that the Lemma holds when $\textnormal{dist}(\sigma, {\sigma}')=1$ and at time $t=1$. We may assume $\sigma\leq\sigma'$ with $\sigma(v)=-1$, $\sigma'(v)=1$ for some vertex $v$, and $\sigma(w)=\sigma'(w)$ elsewhere $w\neq v$. Then, for $i=1,2,\cdots,m$,
\[
\textnormal{dist}({\sigma_1}^{(i)}, {\sigma_1'}^{(i)}) \leq 
\mathbbm{1}_{v\in A_i}(1-\mathbbm{1}_{I=v}+\mathbbm{1}_{I\neq v}\mathbbm{1}_{I\in A_i}\mathbbm{1}_{B_i})
+\sum_{j\neq i} \mathbbm{1}_{v\in A_j} \mathbbm{1}_{I\in A_i}\mathbbm{1}_{B_j},
\]
where
\[
B_i = \left\{ r_+\left(\sum_{w:w\sim I} K(I,w)\sigma(w)\right)\leq U < r_+\left(\sum_{w:w\sim I} K(I,w)\sigma'(w)\right) \right\}.
\]
Note that for $k=1,2,\cdots,m$,
\begin{align*}
\mathbb{P}\big(B_k\big |\big\{I\in A_i & \big\}\big)
= \frac{1}{2} \left(\tanh{\beta\left(\sum_{w:w\sim I} K(I,w)\sigma'(w)\right)} - \tanh{\beta\left(\sum_{w:w\sim I} K(I,w)\sigma(w)\right)}\right)\\
&= \frac{1}{2} \left(\tanh{\beta\left(\sum_{w:w\sim I} K(I,w)\sigma(w)+2K_{ik}\right)} - \tanh{\beta\left(\sum_{w:w\sim I} K(I,w)\sigma(w)\right)}\right)\\
&\leq \tanh{(\beta K_{ik})}\leq \beta K_{ik}.
\end{align*}
Thus,
\[
\mathbb{E}\textnormal{dist}({\sigma_1}^{(i)},{\sigma'_1}^{(i)}) \leq \mathbbm{1}_{v\in A_i} \left(1 - \frac{1}{n} + (p_i - \frac{1}{n}) \beta K_{ii}\right) + \sum_{j\neq i} \mathbbm{1}_{v\in A_j} p_i \beta K_{ij}.
\]
Now consider $\sigma,\sigma'$ with $\textnormal{dist}(\sigma,\sigma')=k$ for some $k>1$. There exist a sequence of configurations $\sigma^0:=\sigma, \sigma^1, \cdots, \sigma^k :=\sigma'$ such that $\textnormal{dist}(\sigma^i,\sigma^{i+1})$=1. By the triangular inequality, we obtain
\[
\mathbb{E}\textnormal{dist}(\sigma_1^{(i)},\sigma_1'^{(i)})\leq \textnormal{dist}(\sigma^{(i)},\sigma'^{(i)})\left(1-\frac{1}{n}+(p_i-\frac{1}{n})\beta K_{ii}\right)+\sum_{j\neq i}\textnormal{dist}(\sigma^{(j)},\sigma'^{(j)})p_i\beta K_{ij}.
\]
This establishes Lemma \ref{L_contr} for $t=1$. Iterating completes the proof.
\end{proof}

Recall from $\eqref{eq:def_Q}$ that $\mathbf{Q}$ is a positive matrix. Thus, by the Perron--Frobenius theorem, there exists the largest eigenvalue $\rho_n>0$ with the left eigenvector $\ba^{\top}:=(a_1,a_2,\cdots,a_m)>\mathbf{0}$ with $\Vert \ba\Vert_1=1$. Moreover, the left eigenvector of the largest eigenvector of $\mathbf{B}$ is also $\ba^{\top}$. Therefore, $\ba^{\top}$ only depends on $p_i$s and $k_{ij}$s. Now we define
\begin{align}\label{const}
\alpha:=\frac{1}{2n(1-\rho_n)},\quad \beta_{cr}:=\frac{1}{\sum_{1\leq i,j\leq m}{a_i p_i k_{ij}}}.
\end{align}

We prove that $\alpha$ and $\beta_{cr}$ are independent of $n$.
\begin{proposition}\label{prop:eigen}
The critical inverse temperature $\beta_{cr}$ is determined by the proportion of the groups and their interaction strengths, i.e., $\beta_{cr}=\beta_{cr}(p_i, k_{ij})$, and $\alpha$ is a constant that depends on $\beta$ but independent of $n$, i.e., $\alpha=\alpha(p_i,k_{ij},\beta)$. Moreover, $\beta < \beta_{cr}$ is equivalent to $\rho_n<1$.
\end{proposition}
\begin{proof}
Considering that $\ba^{\top}$ is independent of $n$, $\beta_{cr}$ is independent of $n$ and is a function on $p_i$ $(1\leq i\leq m)$, and $k_{ij}$ $(1\leq i,j\leq m)$ directly follows from the definition. In addition,
\begin{align*}
    1-\frac{1}{2n\alpha}=\rho_n
    =\Vert \mathbf{Q}^{\top}\ba\Vert_1
    =\left\|\left(\left(1-\frac{1}{n}\right)\mathbf{I}_n+\frac{\beta}{n}\mathbf{B}^{\top}\right)\ba\right\|_1 = 1-\frac{1}{n}+\frac{\beta}{n}\sum_{1\leq i,j\leq m}a_ip_ik_{ij},
\end{align*}
to obtain
\[
\alpha=\frac{1}{2\left(1-\beta/\beta_{cr}\right)},
\]
which shows that $\alpha$ depends only on $\beta$, $p_i$s, and $k_{ij}$s. Finally, by the definition of $\beta_{cr}$, $\beta < \beta_{cr}$ is equivalent to $\alpha>0$ and $\rho_n<1$.
\end{proof}

We may denote the maximum and minimum element of $\ba$ by $a_{\text{max}}:=\max_{1\leq i\leq m}{a_i}$ and $a_{\text{min}}:=\min_{1\leq i\leq m}{a_i}$, that are independent of $n$. The following lemma is required to take advantage of the matrix part of Lemma \ref{L_contr}. We write $\bm{e}_j$ for $j$th unit vector and $\mathbbm{1}$ for all-ones vector on $\mathbb{R}^m$.

\begin{lemma}\label{lemma:Sym}
For $\mathbf{0}\leq \mathbf{s}$, we have
\begin{align*}
    \|\mathbf{Q}^t\bs\|_1\leq \rho_n ^t\left(\displaystyle\sum_{i=1}^m\frac{(s^{(i)})^2}{p_i}\right)^{1/2},\quad
    \bm{e}_j^{\top}\mathbf{Q}^t\bs\leq\sqrt{p_j}\rho_n ^t\left(\displaystyle\sum_{i=1}^m\frac{(s^{(i)})^2}{p_i}\right)^{1/2}.
\end{align*}
\end{lemma}
\begin{proof}
By the definition of the matrices, we get
\[
\mathbf{D}^{-1}\mathbf{Q}\mathbf{D}=
\left(1-\frac{1}{n}\right) \mathbf{I}_n +\frac{\beta}{n}\left((\sqrt{p_ip_j}k_{ij})_{ij}\right)=:\mathbf{X},
\]
which means that the matrix $\mathbf{Q}$ is similar to the symmetric matrix $\mathbf{X}$. Since $\mathbf{X}$ is symmetric, $\Vert\mathbf{X}\Vert_2=\rho_n $. Therefore, we obtain
\begin{align*}
     \|\mathbf{Q}^t\bs\|_1&=\mathbbm{1}^{\top}\mathbf{Q}^t\bs=\mathbbm{1}^{\top}\mathbf{D}\cdot\mathbf{X}^t\mathbf{D}^{-1}\bs=\text{Tr}(\mathbf{X}^t\mathbf{D}^{-1}\bs\mathbbm{1}^{\top}\mathbf{D})\\
     &\leq \|\mathbf{X}^t\mathbf{D}^{-1}\bs\mathbbm{1}^{\top}\mathbf{D}\|_2\leq \|\mathbf{X}\|_2^t\|\mathbf{D}^{-1}\bs\|_2\|\mathbf{D}\mathbbm{1}\|_2=\rho_n ^t\left(\displaystyle\sum_{i=1}^m\frac{(s^{(i)})^2}{p_i}\right)^{1/2},
\end{align*}
where $\text{Tr}(\cdot )$ is the trace of matrix. In a similar manner,
\begin{align*}
    \bm{e}_j^{\top}\mathbf{Q}^t\bs &=\bm{e}_j^{\top}\mathbf{D}\cdot\mathbf{X}^t\mathbf{D}^{-1}\bs
    =\text{Tr}(\mathbf{X}^t\mathbf{D}^{-1}\bs\bm{e}_j^{\top}\mathbf{D})\leq\Vert\mathbf{X}^t\mathbf{D}^{-1}\bs\bm{e}_j^{\top}\mathbf{D}\Vert_2\\
    &\leq\Vert\mathbf{X}\Vert_2^t\Vert\mathbf{D}^{-1}\bs\bm{e}_j^{\top}\mathbf{D}\Vert_2
    \leq \rho_n ^t\Vert\mathbf{D}^{-1}\bs\Vert_2\Vert\mathbf{D}\bm{e}_j\Vert_2
    =\sqrt{p_j}\rho_n ^t\left(\displaystyle\sum_{i=1}^m\frac{(s^{(i)})^2}{p_i}\right)^{1/2}.
\end{align*}
\end{proof}

Combining Lemma \ref{L_contr} and Lemma \ref{lemma:Sym}, we bound the expectation of distance with the initial value without the matrix form.
\begin{lemma}\label{DistLem}
For a monotone coupling $(\sigma_t,\sigma_t')_{t\geq 0}$ starting at $(\sigma,\sigma')$, we have
\[
\E\sum_{i=1}^m a_i\textnormal{dist}(\sigma_t^{(i)},\sigma_t'^{(i)})\leq \rho_n ^t \sum_{i=1}^m a_i \textnormal{dist}(\sigma^{(i)},\sigma'^{(i)}).
\]
In addition, $\E\textnormal{dist}(\sigma_t^{(i)},\sigma_t'^{(i)})\leq n\sqrt{p_i}\rho_n ^t$.
\end{lemma}
\begin{proof}
The first equation is directly derived from multiplying $\ba^{\top}$ on both sides of Lemma \ref{L_contr}. Since $\textnormal{dist}(\sigma_t^{(i)},\sigma_t'^{(i)})\leq np_i$ for all $i$, by multiplying $\bm{e}_i^{\top}$ on both sides of Lemma \ref{L_contr}, we get
\[
\E\textnormal{dist}(\sigma_t^{(i)},\sigma_t'^{(i)})\leq n\bm{e}_i^{\top}\mathbf{Q}^t(p_1,p_2,\cdots,p_m)^{\top}\leq n\sqrt{p_i}\rho_n ^t.
\]
The last inequality comes from Lemma $\ref{lemma:Sym}$.
\end{proof}

\subsection{Magnetization chain}\label{subsec:magnetization chains}
Recall the definition of the magnetization $\bS(\sigma)$ from \eqref{eq:def_mag}. We use the notation $\bS_t=\bS(\sigma_t)$ and $\bm{M}_t=\bm{M}(\sigma_t)$ throughout this paper. We first show that the magnetization chain is a Markov chain. Indeed, $\bS_{t+1}-\bS_t$ takes values in $\cup_{i=1}^m\{-\frac{2}{n}\bm{e}_i, +\frac{2}{n}\bm{e}_i, 0\}$, where $\bm{e}_i$ is the $i$th unit vector on $\mathbb{R}^m$.
\begin{proposition}\label{prop:Markov}
The process $\bS_t=(S_t^{(1)},\cdots,S_t^{(m)})$ is a Markov chain on the magnetization state space $\mathcal{S}$, where
\[
\mathcal{S}^{(i)}:=\{-p_i,-p_i+2/n,\cdots,p_i-2/n,p_i\},\quad \mathcal{S}:=\prod_{i=1}^m\mathcal{S}^{(i)}.
\]
\end{proposition}
\begin{proof}
Clearly, each of $S_t^{(i)}$ can take a value in the set $\{-p_i,-p_i+2/n,\cdots,p_i-2/n,p_i\}$. Furthermore, it is easy to examine that
\begin{align}\label{eq:eqeq}
    \prob\left(\bS_{t+1}=\bS_t+\frac{2}{n}\bm{e}_i\right)&=p_i\cdot \frac{(|G_i|-nS_t^{(i)})/2}{|G_i|}\cdot r_+\left(\sum_{j\neq i}k_{ij}S_t^{(j)}+k_{ii}\left(S_t^{(i)}+\frac{1}{n}\right)\right).
\end{align}
Now, we define
\begin{align}\label{eq:def_X}
    X_t^{(i)}:=\sum_{j=1}^m k_{ij}S_t^{(j)}.
\end{align}
Then, we can rewrite the aforementioned equation \eqref{eq:eqeq} and obtain the rest of the equations in a similar way,
\begin{align*}
    &\prob\left(\bS_{t+1}=\bS_t\pm\frac{2}{n}\bm{e}_i\right)=\frac{p_i\mp S_t^{(i)}}{2}\cdot r_{\pm}\left(X_t^{(i)}\pm\frac{k_{ii}}{n}\right),\\
    &\prob\left(\bS_{t+1}=\bS_t\right)
    =\sum_{i=1}^{m}{\left(p_i-\frac{p_i-S_t^{(i)}}{2}\cdot r_+\left(X_t^{(i)}+\frac{k_{ii}}{n}\right)-\frac{p_i+S_t^{(i)}}{2}\cdot r_-\left(X_t^{(i)}-\frac{k_{ii}}{n}\right)\right)}.
\end{align*}
\end{proof}
\begin{remark}\label{Rmk:UniformlyBdd}
Observing that 
\[
\left|X_t^{(i)}\pm\frac{k_{ii}}{n}\right|\leq \sum_{j=1}^m k_{ij}p_j+ \frac{k_{ii}}{n}\leq \sum_{j=1}^m k_{ij}p_j+\max_{1\leq i\leq m}{k_{ii}}=:C_k,
\]
and $C_k$ is a constant independent of $n$, the derivative of $\tanh x$ is bounded away from both $0$ and $1$ for every possible $x$ in the state space. Thus, by the mean value theorem, we have
\begin{alignat*}{2}
    &\prob\Big(\bS_{t+1}&&=\bS_t\pm  \frac{2}{n}\bm{e}_i\Big)=\frac{p_i\mp S_t^{(i)}}{4}\cdot\left(1\pm\tanh(X_t^{(i)})\right)+O(n^{-1}),\\
    &\prob\Big(\bS_{t+1}&&=\bS_t \Big)\\ 
    & &&=\sum_{i=1}^{m}{\left(p_i-\frac{p_i-S_t^{(i)}}{4}\cdot\left(1+\tanh(X_t^{(i)})\right)-\frac{p_i+S_t^{(i)}}{4}\cdot \left(1-\tanh(X_t^{(i)})\right)\right)}+O(n^{-1})\\
    & && =\sum_{i=1}^{m}{\left(\frac{p_i}{2}+\frac{S_t^{(i)}}{2}\tanh(X_t^{(i)}))\right)}+O(n^{-1})=\frac{1}{2}+\sum_{i=1}^{m}{\left(\frac{S_t^{(i)}}{2}\tanh(X_t^{(i)}))\right)}+O(n^{-1}).
\end{alignat*}
Thus, we have that the remaining probability is uniformly bounded away from $0$ and $1$.
\end{remark}
Continuing the remark, we can also observe the drift of $S_t^{(i)}$. We define a $\sigma$-algebra $\mathcal{F}_t$ generated by $S_t^{(1)},S_t^{(2)},\cdots,S_t^{(m)}$. By direct calculation, we obtain
\begin{equation}\label{eq:drift}
\begin{split}
    \mathbb{E}\left[S_{t+1}^{(i)}-S_t^{(i)}|\mathcal{F}_t\right]
    &= \frac{2}{n}\cdot \frac{p_i -S_t^{(i)}}{2}\cdot r_+\left(X_t^{(i)}+K_{ii}\right)
    -\frac{2}{n}\cdot \frac{p_i +S_t^{(i)}}{2}\cdot r_-\left(X_t^{(i)}-K_{ii}\right)\\
    &=\frac{1}{n}\left( -S_t^{(i)} +f_n({\bS}_t)+g_n({\bS}_t) \right),
\end{split}
\end{equation}
where
\begin{align*}
    f_i({\bS}_t)&=p_i\left(r_+\left(X_t^{(i)}+K_{ii}\right)-r_-\left(X_t^{(i)}-K_{ii}\right)\right),\\
    g_i({\bS}_t)&=-S_t^{(i)}\left(r_+\left(X_t^{(i)}+K_{ii}\right)+r_-\left(X_t^{(i)}-K_{ii}\right)-1\right).
\end{align*}
Note that with the similar technique in Remark \ref{Rmk:UniformlyBdd},
\begin{equation}\label{control_r}
\begin{split}
    r_+\left(X_t^{(i)}+K_{ii}\right)-r_-\left(X_t^{(i)}-K_{ii}\right)
    &=\frac{1}{2}\left(\tanh{\beta(X_t^{(i)}+K_{ii})}+\tanh{\beta(X_t^{(i)}-K_{ii}})\right)\\
    &=\tanh{\beta X_t^{(i)}}+O(n^{-1}),\\
    r_+\left(X_t^{(i)}+K_{ii}\right)+r_-\left(X_t^{(i)}-K_{ii}\right)
    &=1-\frac{1}{2}\left(\tanh{\beta(X_t^{(i)}+K_{ii})}-\tanh{\beta(X_t^{(i)}-K_{ii}})\right)\\
    &=1+O(n^{-1}),
\end{split}
\end{equation}
to obtain the approximation
\begin{align*}
    \mathbb{E}\left[S_{t+1}^{(i)}-S_t^{(i)}|\mathcal{F}_t\right]
    \sim \frac{1}{n}\left( -S_t^{(i)} +p_i\tanh{(\beta X_t^{(i)}}) \right).
\end{align*}
We will scrutinize this drift later.

Now we demonstrate the contraction properties in terms of the magnetization chain.
\begin{lemma}\label{lemma:MagCont}
For starting magnetizations $\bs:=(s^{(1)},\cdots,s^{(m)})\geq \bs':=(s'^{(1)},\cdots,s'^{(m)})$, the magnetization satisfies
\[
0\leq
\begin{pmatrix}
\mathbb{E}_{\bs}S_t^{(1)}-\E_{\bs'}S_t'^{(1)}\\
\mathbb{E}_{\bs}S_t^{(2)}-\E_{\bs'}S_t'^{(2)}\\
\vdots \\
\mathbb{E}_{\bs}S_t^{(m)}-\E_{\bs'}S_t'^{(m)}\\
\end{pmatrix}
\leq \mathbf{Q}^t
\begin{pmatrix}
s^{(1)}-s'^{(1)}\\
s^{(2)}-s'^{(2)}\\
\vdots \\
s^{(m)}-s'^{(m)}\\
\end{pmatrix}.
\]
\end{lemma}
\begin{proof}
Let $(\sigma_t,{\sigma'}_t)$ be a monotone coupling starting from $(\sigma,\sigma')$ with $\sigma\geq\sigma'$, where the starting magnetizations satisfy $\bS(\sigma)=\bs$ and $\bS(\sigma')=\bs'$. By monotonicity, $\sigma_t^{(i)}\geq{\sigma'}_t^{(i)}$ holds for all $i$ and $t$. Therefore, $S_t^{(i)}-{S'}_t^{(i)}=|S_t^{(i)}-{S'}_t^{(i)}|=\frac{2}{n}\textnormal{dist}(\sigma_t^{(i)},{\sigma'}_t^{(i)})\geq 0$, and by Lemma \ref{L_contr},
\[
0\leq
\begin{pmatrix}
\mathbb{E}_{\sigma}S_t^{(1)}-\E_{\sigma'}S_t'^{(1)}\\
\mathbb{E}_{\sigma}S_t^{(2)}-\E_{\sigma'}S_t'^{(2)}\\
\vdots \\
\mathbb{E}_{\sigma}S_t^{(m)}-\E_{\sigma'}S_t'^{(m)}\\
\end{pmatrix}=
\begin{pmatrix}
\mathbb{E}_{\sigma,\sigma'}|S_t^{(1)}-S_t'^{(1)}|\\
\mathbb{E}_{\sigma,\sigma'}|S_t^{(2)}-S_t'^{(2)}|\\
\vdots \\
\mathbb{E}_{\sigma,\sigma'}|S_t^{(m)}-S_t'^{(m)}|\\
\end{pmatrix}
\leq \mathbf{Q}^t
\begin{pmatrix}
s^{(1)}-s'^{(1)}\\
s^{(2)}-s'^{(2)}\\
\vdots \\
s^{(m)}-s'^{(m)}\\
\end{pmatrix}.
\]
The conclusion follows from Proposition \ref{prop:Markov} that $\mathbb{E}_{\sigma}S_t^{(i)}-\E_{\sigma'}S_t'^{(i)}=\mathbb{E}_{\bs}S_t^{(i)}-\E_{\bs'}S_t'^{(i)}$.
\end{proof}
Recall that the notation $\circ$ indicates the Hadamard product, i.e., the element-wise product.
\begin{proposition}\label{VariProp}
For a monotone coupling $(\sigma_t,\sigma'_t)_{t\geq 0}$ starting at $(\sigma,\sigma')$ with magnetizations $(\bs,\bs')$,
\[
\E_{\sigma,\sigma'}\Vert \ba\circ{\bS}_t-\ba\circ\mathbf{S'}_t\Vert_1\leq \rho_n^t\Vert \ba\circ\mathbf{s}-\ba\circ\mathbf{s'}\Vert_1.
\]
\end{proposition}
\begin{proof}
For any starting magnetizations $\bs$ and $\bs'$, there exists $\bs_0:=\bs$, $\bs_1,\cdots,\bs_{k-1}$, $\bs_k:=\bs'$ such that $\bs_{i-1}-\bs_{i}=\bm{e}_i(s^{(i)}-s'^{(i)})$ for $i=1,2,\cdots,k$. Then, we can consider monotone coupling $\sigma_{0,t}, \sigma_{1,t},\cdots,\sigma_{k,t}$ with the corresponding magnetization of the starting configuration as $\bs_0,\bs_1,\cdots,\bs_k$. Note that either $\bs_{i-1}\leq\bs_i$ or $\bs_{i-1}\geq\bs_i$ holds; therefore,  we can apply Lemma \ref{lemma:MagCont}.
Let $\bS_{i,t}$ denote the magnetization with respect to $\sigma_{i,t}$ for $i=0,1,\cdots,k$. Then, with the triangle inequality and Lemma \ref{lemma:MagCont},
\begin{align*}
    \E_{\sigma,\sigma'}\Vert\ba\circ\bS_t-\ba\circ  \bS_t' & \Vert_1
    \leq \sum_{i=1}^m\E_{\sigma_{i-1},\sigma_i}\Vert\ba\circ\bS_{i-1,t}-\ba\circ\bS_{i,t}\Vert_1\\
    &\leq\sum_{i=1}^{m}{\ba^{\top}\mathbf{Q}^t\bm{e}_i|s^{(i)}-s'^{(i)}|}
    =\rho_n ^t\sum_{i=1}^m a_i|s^{(i)}-s'^{(i)}|=\rho_n ^t\Vert\ba\circ\bs-\ba\circ\bs'\Vert_1.
\end{align*}
\end{proof}

\subsection{Variance Bound}\label{subsec:variance bound}
In this subsection, we analyze the variance bound of the magnetization. We start with the following lemma stated in \cite[Lemma 3.1.]{Hee}. 
\begin{lemma}\label{VariLem}
Let $(\bm{Z}_t)_{t\geq 0}$ be a Markov chain taking values in $\mathbb{R}^m$ with the expectation $\E_\bs$ when $\bm{Z}_0=\bs$. Suppose that there exist some $0<\kappa<1$ such that for any starting states $\bs,\bs'\in \mathbb{R}^m$,
\[
\Vert {\mathbb{E}_\bs}{\bm{Z}}_t - \mathbb{E}_{\bs'}{\bm{Z}}_t\Vert_1\leq \kappa^t\Vert \bs-\bs'\Vert_1.
\]
Then, for the $l^2$ norm variance $v_t:=\sup_{\bs_0}\mathrm{Var}_{\bs_0}\bm{Z}_t$,
\[
v_t \leq m v_1 \cdot \min\{t,(1-\kappa^2)^{-1}\}.
\]
\end{lemma}
We employ this lemma to prove the upper bound on the variance of magnetization.
\begin{proposition}\label{prop:VariBound}
When $\beta<\beta_{cr}$, for any starting state $\bs$, we have $\displaystyle\sum_{i=1}^m{{\Var_\bs}(S_t^{(i)})}=O(n^{-1})$ and ${\Var_\bs} \Vert\bS_t\Vert_1=O(n^{-1})$. When $\beta=\beta_{cr}$, we have ${\Var_\bs} \|\bS_t\|_1=O(t/n^2)$.
\end{proposition}
\begin{proof}
Note that the increments of $S_1^{(i)}$ from $S^{(i)}$ are bounded by $2/n$ in absolute value for all $1\leq i\leq m$. Therefore,
\[
\sum_{i=1}^m{\Var_\bs a_iS_1^{(i)}\leq a_{max}^2(2/n)^2}.
\]
Recall that $\beta<\beta_{cr}$ is equivalent to $\rho_n <1$. Thus, by Proposition \ref{VariProp} and Lemma \ref{VariLem}, we obtain
\[
a_{min}^2\sum_{i=1}^{m}\Var_\bs(S_t^{(i)})\leq \sum_{i=1}^m\Var_\bs(a_i S_t^{(i)})\leq m \frac{4a_{max}^2}{n^2}\frac{1}{1-\rho_n ^2}=\frac{8\alpha m a_{max}^2}{n(1+\rho_n )}\leq \frac{8\alpha m a_{max}^2}{n},
\]
to have $\sum_{i=1}^m{\Var_\bs(S_t^{(i)})}=O(n^{-1})$. In addition,
\begin{align*}
    \Var_\bs\Vert\bS_t\Vert_1
    &=\sum_{i=1}^m\Var_\bs|S_t^{(i)}|+2\sum_{i<j}\text{Cov}(|S_t^{(i)}|,|S_t^{(ij}|)\\
    &\leq \sum_{i=1}^m\Var_\bs S_t^{(i)}+2\sum_{i<j}\sqrt{\Var_\bs S_t^{(i)}}\sqrt{\Var_s S_t^{(j)}}\\
    &\leq \sum_{i=1}^m\Var_\bs S_t^{(i)}+\sum_{i<j}(\Var_\bs S_t^{(i)}+\Var_\bs S_t^{(j)})
    =m\sum_{i=1}^m\Var_\bs S_t^{(i)} = O(n^{-1}).
\end{align*}
Moreover, $\beta=\beta_{cr}$ implies $\rho_n =1$. Thus, again by Proposition \ref{VariProp} and Lemma \ref{VariLem}, we obtain
\[
a_{min}^2\sum_{i=1}^{m}\Var_\bs(S_t^{(i)})\leq \sum_{i=1}^m\Var_\bs(a_i S_t^{(i)})\leq m \frac{4a_{max}^2}{n^2}t=O\left(\frac{t}{n^2}\right).
\]
and $\Var_\bs\|\bS_t\|_1=O(t/n^2)$ in a similar way. 
\end{proof}

Furthermore, we examine the expected number of positive and negative spins on subsets of $G_i$ in the high temperature regime to establish the cutoff.

\begin{proposition}\label{Prop:expected}
When $\beta<\beta_{cr}$, for any $B\subseteq G_i$, and a chain $(\sigma_t)_{t\geq 0}$ starting at $\sigma$, let us define
\[
M_t(B):=\frac{1}{2}\sum_{v\in B}\sigma_t{(v)}.
\]
Then the following hold:
\begin{enumerate}
    \item $|\E_\sigma M_t(B)|\leq |B|\rho_n ^t\sqrt{p_i}$,
    \item $\Var_\sigma(M_t(B))=O(n)$, and $\E_\sigma|M_t(B)|\leq |B|\rho_n ^t\sqrt{p_i}+O(\sqrt{n})$,
    \item For any starting magnetization $\bs\geq 0$,
    \begin{align*}
    \E_\bs\Vert\bS_t\Vert_1&\leq \rho_n ^t\left(\sum_{i=1}^m\frac{(s^{(i)})^2}{p_i}\right)^{1/2}+O(n^{-1/2}) ,\quad \E_\bs\Vert\ba^{\top}\bS_t\Vert_1\leq \rho_n ^t\ba^{\top}\bs+O(n^{-1/2}),
\end{align*}
\begin{align*}
    \text{and}\quad 0\leq\E_\bs S_i^{(i)}\leq \sqrt{p_i}\rho_n ^t\left(\sum_{i=1}^m\frac{s^{(i)}}{p_i}\right)^{1/2}.
\end{align*}
\end{enumerate}
\end{proposition}
\begin{proof}
Let $\mathbf{1}$ denote the configuration of all positive spins, and let $(\sigma_t^{\mathbf{1}}, \sigma_t^{\mu})$ be the monotone coupling with starting configuration $(\mathbf{1},\mu)$, where $\mu$ is the stationary distribution. Then, by Lemma \ref{DistLem} and triangular inequality,
\[
\E_{\mathbf{1}} \left[M^{\bone}_t(G_i)\right] \leq \E_{(\mathbf{1},\mu)}\left[|M^{\bone}_t(G_i)-M^{\mu}_t(G_i)|\right]+\E_{\mu}\left[M^{\mu}_t(G_i)\right]\leq n\sqrt{p_i}\rho_n ^t,
\]
considering that $\E_{\mu}\left[M^{\mu}_t(G_i)\right]=0$. Hence, by symmetry of vertices on set $G_i$, we have $\E_{\mathbf{1}}\left[ M^{\bone}_t(B)\right]\leq |B|\rho_n ^t/\sqrt{p_i}$. Accordingly, for any starting state $\sigma$, by monotonicity, we obtain $\E_\sigma M_t(B)\leq \E_{\mathbf{1}} M^{\bone}_t(B) \leq|B|\rho_n ^t\sqrt{p_i}$. Similarly, considering the configuration of all negative spins (denoted by $-\bone$), we obtain $\E_\sigma M_t(B)\geq-|B|\rho_n ^t\sqrt{p_i}$. Consequently, $|\E_\sigma M_t(B)|\leq |B|\rho_n ^t\sqrt{p_i}$ for any starting configuration $\sigma$.

For part (ii), because the spins are positively correlated (\cite{correlation}), from Proposition \ref{prop:VariBound},
\[
\Var(\sum_{v\in B}\sigma_t(v))\leq \Var(\sum_{v\in G_i}\sigma_t(v))\leq n^2\Var(S_t^{(i)})\leq cn,
\]
to obtain $\Var(M_t(B))=O(n)$. Moreover, for the expectation of absolute value,
\begin{align*}
    \E_\sigma|M_t(B)|
    &\leq \sqrt{\E_\sigma\left[|M_t(B)|^2\right]}=\sqrt{(\E_\sigma[M_t(B)])^2+\Var_\sigma(M_t(B))}\\
    &\leq |\E_\sigma M_t(B)|+\sqrt{\Var_{\sigma}(M_t(B))}\leq |B|\rho_n ^t\sqrt{p_i}+O(\sqrt{n}).
\end{align*}

For part (iii), consider the vector $\bm{x}:=\left(\frac{1-(-1)^{|G_i|}}{2}\right)_{i=1}^m\in\mathbb{R}^m$, which is the least non-negative starting configuration on $G_i$ and define the starting distribution $\nu$, which yields $\bm{x}$ with probability $1/2$ and $-{\bm{x}}$ with probability $1/2$. For any starting state $\bs\geq \mathbf{0}$, we apply $\bs$ and $\nu$ on Lemma \ref{lemma:MagCont}. Note that $\bs\geq{\bm{x}}\geq-{\bm{x}}$ and by symmetry, $\E_\nu S_t^{(i)}=0$. Hence,
\begin{align}\label{eq1}
0\leq
\begin{pmatrix}
\mathbb{E}_{\bs}S_t^{(1)}-\E_{\nu}S_t'^{(1)}\\
\mathbb{E}_{\bs}S_t^{(2)}-\E_{\nu}S_t'^{(2)}\\
\vdots \\
\mathbb{E}_{\bs}S_t^{(m)}-\E_{\nu}S_t'^{(m)}\\
\end{pmatrix}
=
\begin{pmatrix}
\mathbb{E}_{\bs}S_t^{(1)}\\
\mathbb{E}_{\bs}S_t^{(2)}\\
\vdots \\
\mathbb{E}_{\bs}S_t^{(m)}\\
\end{pmatrix}
\leq \frac{1}{2}\mathbf{Q}^t(\bs-{\bm{x}})+\frac{1}{2}\mathbf{Q}^t(\bs+{\bm{x}})=\mathbf{Q}^t\bs.
\end{align}
Therefore, $\sum_{i=1}^m\E_\bs S_t^{(i)}\leq\Vert\mathbf{Q}^t\bs\Vert_1\leq \rho_n ^t\left(\sum_{i=1}^m\frac{(s^{(i)})^2}{p_i}\right)^{1/2}$ by Lemma \ref{lemma:Sym}. Finally,
\begin{align*}
    \E_\bs\Vert\bS_t\Vert_1
    &=\sum_{i=1}^m\E_\bs{|S_t^{(i)}|}
    \leq\sum_{i=1}^m\sqrt{\E_\bs{|S_t^{(i)}|^2}}
    =\sum_{i=1}^m\sqrt{(\E_\bs{S_t^{(i)}})^2+\Var_\bs{S_t^{(i)}}}\\
    &\leq \sum_{i=1}^m\left(|\E_\bs{S_t^{(i)}}|+\sqrt{\Var_\bs{S_t^{(i)}}}\right)
    =\sum_{i=1}^m\left(\E_\bs{S_t^{(i)}}+\sqrt{\Var_\bs{S_t^{(i)}}}\right)\\
    &\leq \rho_n ^t\left(\sum_{i=1}^m\frac{(s^{(i)})^2}{p_i}\right)^{1/2}+O(n^{-1/2}).
\end{align*}
Moreover, by multiplying $\ba^{\top}$ on both sides of \eqref{eq1}, we obtain
\begin{align}\label{eq2}
    0\leq\E_{\bs}\left[\ba^{\top}\bS_t\right] \leq \rho_n ^t\ba^{\top}\bs.
\end{align}
Similarly, combining Proposition \ref{prop:VariBound} and (\ref{eq2}),
\begin{align*}
    \E_\bs\Vert\ba^{\top}\bS_t\Vert_1
    &=\sum_{i=1}^m\E_\bs{|a_i S_t^{(i)}|}
    \leq\sum_{i=1}^m\sqrt{\E_\bs{|a_i S_t^{(i)}|^2}}\\
    &=\sum_{i=1}^m\sqrt{(\E_\bs{a_i S_t^{(i)}})^2+\Var_\bs{a_i S_t^{(i)}}}
    \leq \sum_{i=1}^m\left(|\E_\bs{a_i S_t^{(i)}}|+\sqrt{\Var_\bs{a_i S_t^{(i)}}}\right)\\
    &=\E_\bs\left[ \ba^{\top}\bS_t \right]+O(n^{-1/2})\leq \rho_n ^t\ba^{\top}\bs+O(n^{-1/2}).
\end{align*}
Now, by multiplying $\bm{e}_i^{\top}$ on both sides of (\ref{eq1}), we obtain,
\[
0\leq\E_\bs S_i^{(i)}\leq \bm{e}_i^{\top}\mathbf{Q}^t\bs\leq \sqrt{p_i}\rho_n ^t\left(\sum_{i=1}^m\frac{s^{(i)}}{p_i}\right)^{1/2}.
\]
where the last inequality follows from Lemma \ref{lemma:Sym}.
\end{proof}

\subsection{Coupling after the same magnetization}\label{subsec:after}
We claim that once the magnetizations agree, $O(n\log n)$ steps are sufficient to obtain full agreement at all temperatures. Note that we only require $O(n)$ steps to have full agreement in the high temperature regime $(\beta<\beta_{cr})$ which will be proved in Section \ref{sec:high}.
\begin{lemma}\label{Lem:CoupleSameMag}
For two configurations $\sigma$, $\sigma'$ satisfying $\bS(\sigma)=\bS(\sigma')$, there is a constant $c(\beta)$ such that
\[
\limsup_{n\rightarrow\infty}\prob_{\sigma,\sigma'}(\tau>c(\beta)n\log n)=0,
\]
where $\tau:=\min\{t\geq 0:\sigma_t=\sigma'_t \}$.
\end{lemma}
\begin{proof}
We construct the coupling as follows. First, we take random variables $I$ uniformly on the vertex set $V$ and $U$ uniformly on $[0,1]$, independently. We then suppose $I=i\in G_j$. For each $\sigma$, we define the spin $S^{I}$ by
\[
S^{I}=
\begin{cases}
+1,& 0<U\leq r_+(\sum_{1\leq l\leq m}{K_{jl}M_l}-K_{jj}\sigma(I)),\\
-1,& r_+(\sum_{1\leq l\leq m}{K_{jl}M_l}-K_{jj}\sigma(I))<U\leq 1,
\end{cases}
\]
and determine the next configuration by
\[
\sigma_1(v)=
\begin{cases}
\sigma(v),& v\neq I,\\
S^{I},& v=I,
\end{cases}
\]
Now, for $\sigma'$, if $\sigma'(I)=\sigma(I)$, we determine the configuration by
\[
\sigma'_1(v)=
\begin{cases}
\sigma(v),& v\neq I,\\
S^{I},& v=I,
\end{cases}
\]
and if $\sigma'(I)\neq\sigma(I)$, we choose $I'$ uniformly on the set $\{i\in G_j : \sigma'(i)=\sigma(I),~\sigma'(i)\neq\sigma(i)\},$
and determine the configuration by
\[
\sigma'_1(v)=
\begin{cases}
\sigma(v),& v\neq I',\\
S^{I},& v=I'.
\end{cases}
\]
Note that the site is chosen uniformly on $V$, and this is also the Glauber dynamics. Recall the $i$-th Hamming distance from \eqref{eq:distance} and define\\[-0.3cm]
\[
d^{(i)}_t:=\textnormal{dist}(\sigma^{(i)}_t,\sigma'^{(i)}_t),\quad d_t:=\sum_{i=1}^m{d^{(i)}_t}.
\]\\[-0.3cm]
As mentioned in Remark \ref{Rmk:UniformlyBdd}, $r_+(s)$ and $r_-(s)$ are bounded below uniformly for $s$ in the domain; hence, we have a constant $c_1(\beta)$ such that for all possible $s$, $\min\{r_+(s),r_-(s)\}\geq c_1(\beta)$. Because $d^{(i)}$ decreases by $-1$ if $\sigma'(I)\neq\sigma(I)$ and $S^I=\sigma(I')=\sigma'(I)$, and otherwise remains the same, we have
\begin{align*}
\E\left[d^{(i)}_{t+1}-d^{(i)}_t|\mathcal{F}_t\right]\leq -c_1(\beta)\frac{d^{(i)}_t}{n},
\end{align*}
to get $\E[d_{t+1}|\mathcal{F}_t]\leq -\left(1-\frac{c_1(\beta)}{n}\right)d_t$ by summing from $i=1$ to $i=m$. Hence $d_t\left(1-\frac{c_1(\beta)}{n}\right)^{-t}$ is a non-negative supermartingale, which indicates that
\[
\E[d_t]\leq\E[d_0]\left(1-\frac{c_1(\beta)}{n}\right)^{t}\leq \E[d_0]e^{-c_1(\beta)t/n}\leq ne^{-c_1(\beta)t/n}.
\]
There exist a sufficiently large constant $c(\beta)$ that makes the right-hand side less than $n^{-1}$ when $t\geq c(\beta)n\log n$. Hence,
\[
\prob_{\sigma,\sigma'}(\tau>c(\beta)n\log n)\leq \prob_{\sigma,\sigma'}(d_{c(\beta)n\log n}\geq 1)\leq \E_{\sigma,\sigma'}(d_{c(\beta)n\log n})< n^{-1}.
\]
\end{proof}

\subsection{Upper bound of stopping time}\label{subsec:superMG}
In closing, the following lemma is helpful to show the upper bound of the mixing time in the high temperature regime ($\beta\leq\beta_{cr}$) because certain functions for magnetization satisfies supermartingale property.
\begin{lemma}\label{Lem:SupMar}
Let $(W_t)_{t\geq 0}$ be a non-negative supermartingale with a stopping time $\tau$ satisfying
\begin{enumerate}
    \item $W_0=k$,
    \item $|W_{t+1}-W_t|\leq B<\infty$ for some constant $B$,
    \item $\Var(W_{t+1}|\mathcal{F}_t)\geq \sigma^2>0$ on the event $\{\tau>t\}$.
\end{enumerate}
Then, for $u>12B^2/\sigma^2$, $\prob_k(\tau>u)\leq \frac{4k}{\sigma \sqrt{u}}$ holds.
\end{lemma}
\begin{remark}
We apply this lemma when $\mathcal{F}_t$ is determined by $W_t$. In this case, from the definition of conditional variance, the third condition is satisfied when $\prob(W_{t+1}\neq W_t|\mathcal{F}_t)$ is uniformly bounded above on the event $\{\tau>t\}$, and $|W_{t+1}-W_{t}|$ is bounded away from zero on the event $\{W_{t+1}\neq W_{t}\}$. Precisely,
\begin{align*}
    \prob(W_{t+1}\neq W_t|\mathcal{F}_t)\cdot \min\{(W_{t+1}-W_t)^2:W_{t+1}\neq W_t\}\geq \sigma^2>0 \text{ on the event } \{\tau>t\}.
\end{align*}
\end{remark}
The proof can be found in \cite{MCMC}: however, we prove the Lemma for the self-containedness of the article.
\begin{proof}
    By Doob's decomposition, we can write $W_t=M_t-A_t$ with martingale $(M_t)$ with $M_0=W_0$ and non-decreasing previsible $(A_t)$ with $A_0=0$. Then, the second condition gives
    \[
    A_{t+1}-A_t=-\E[W_{t+1}-W_t|\mathcal{F}_t]\leq B,
    \]
    which implies that $M_{t+1}-M_t\leq 2B$. In addition, by the third condition, since $A_t$ is previsible, on the event $\{\tau>t\}$, we have
    \begin{align*}
    \Var[M_{t+1}|\mathcal{F}_t]
    &=\E[(M_{t+1}-\E[M_{t+1}|\mathcal{F}_t])^2]|\mathcal{F}_t]\\
    &=\E[(W_{t+1}-\E[W_{t+1}|\mathcal{F}_t])^2]|\mathcal{F}_t]
    =\Var[W_{t+1}|\mathcal{F}_t]\geq\sigma^2>0.
    \end{align*}
    Now, for arbitrary fixed $h>2B$, we define the stopping time $T_h:=\min\{t:M_t\geq h\}\wedge\tau\wedge u.$ Since $T_h$ is bounded by $u$, we can apply the optional stopping theorem to get
    \begin{align}\label{L0}
        k=\E[M_{T_h}]\geq h\prob(M_{T_h}\geq h). 
    \end{align}
    Now, we define a process by $V_t:=M_t^2-hM_t-\sigma^2 t$. Note that
    \begin{align*}
        \E[V_{t+1}-V_{t}|\mathcal{F}_t]
        =\E[(M_{t+1}-M_t)^2|\mathcal{F}_t]-\sigma^2\geq0,
    \end{align*}
    which means $V_{t\wedge \tau}$ is a submartingale. Additionally, using the optional stopping theorem, we obtain
    \[
    -kh\leq V_0\leq \E[V_{T_h}]=\E[M_{T_h}(M_{T_h}-h))]-\sigma^2 \E[T_h].
    \]
    Note that both $W_t$ and $A_t$ are non-negative, as well as $M_t$, which implies that $M_{T_h}(M_{T_h}-h)$ is positive iff $M_{T_h}> h$. Hence, combining the fact that the increment of $M_t$ is at most $2B$ and \eqref{L0}, we bound the first term of the right-hand side by
    \[
    \E[M_{T_h}(M_{T_h}-h))]\leq \prob(M_{T_h}> h)(h+2B)(2B)\leq \prob(M_{T_h}> h)\cdot 2h^2\leq 2kh,
    \]
    which concludes that $\E[T_h]\leq 3kh/\sigma^2$; thus, $\prob(T_h\geq u)\leq \frac{3kh}{\sigma^2 u}$. Therefore, combining again with \eqref{L0},
    \[
    \prob(\tau>u)\leq \prob(M_{T_h}\geq h)+\prob(T_h\geq u)\leq \frac{k}{h}+\frac{3kh}{\sigma^2 u}.
    \]
    Finally, we consider $h=\sqrt{\sigma^2 u/3}$ to obtain that
    \[
    \prob(\tau>u)\leq\frac{2\sqrt{3}k}{\sigma\sqrt{u}}<\frac{4k}{\sigma\sqrt{u}}.
    \]
    Note that we can choose $h>2B$ only when $u>12B^2/\sigma^2$.
\end{proof}

\section{Cutoff at High Temperature}\label{sec:high}
In this section, we prove the cutoff phenomenon in the high temperature regime, when $\beta < \beta_{cr}$. Recall from \eqref{const} that $\alpha= \frac{1}{2n(1-\rho_n)}=\frac{1}{2(1-\beta/\beta_{cr})}$, a constant independent of $n$, and define $t_n := \alpha n\log n$. Without loss of generality, we may assume $p_1\leq p_2\leq\cdots\leq p_m$.
\subsection{Upper Bound}
We may restate the upper bound part of Theorem \ref{Thm:HighTemp}.
\begin{theorem}\label{thm:cutoffUpper}
When $\beta < \beta_{cr}$, then
\[
\lim_{\gamma\rightarrow\infty}\lim_{n\rightarrow\infty} d_n(t_n+\gamma n) =0.
\]
\end{theorem}

To show the upper bound, we first construct a coupling of the dynamics so that the magnetizations agree after $t_n+O(n)$ steps.

\begin{lemma}\label{lem:MagCoup}
For any two configurations $\sigma,\sigma'$, there is a coupling $(\sigma_t,\sigma_t')$ with starting states $(\sigma,\sigma')$ such that, if $\tau_{\text{mag}}:=\min\{t\geq 0:\bS_t=\bS_t'\}$,
\[
\prob_{\sigma,\sigma'}(\tau_{\text{mag}}>t_n+\gamma n)\leq \frac{C}{\sqrt{\gamma}},
\]
holds for sufficiently large $\gamma n$, where $C$ is a constant that is independent of $\sigma,\sigma'$ and $n$.
\end{lemma}
Here, we use \emph{``modified" monotone coupling}. To illustrate, we first define \emph{modified matching}. For $\sigma, \sigma' \in \Omega$ and their magnetization $\bs, \bs'\in\mathcal{S}$, we suppose $s^{(i)}\geq s'^{(i)}$ for some $1\leq i\leq m$, which means that the number of positive spin in $s^{(i)}$ is more than or the same as that of in $s'^{(i)}$. We then match each positive spin in $s^{(i)}$ with a positive spin in $s'^{(i)}$ and match the remaining spins arbitrarily. In this manner, we can construct a bijection $f_{\sigma,\sigma'}:V\rightarrow V'$, which is known as modified matching. We update the matched vertices together in the two chains by modified monotone coupling.

\begin{proof}
Our strategy is as follows: initially, we run the monotone coupling $(\sigma_t,\sigma'_t)$ with starting states $(\sigma,\sigma')$ until time $t_n$. We then define $Y_t:=\frac{n}{2}\Vert \ba\circ\bS_t - \ba\circ\bS'_t \Vert_1$; hence, we have 
\begin{equation}\label{4.1}
    \begin{split}
        \E_{\sigma,\sigma'}[Y_{t_n}]\leq \frac{n}{2} \rho_n^{\alpha n\log{n}}\sum_{i=1}^m(2a_i p_i)=c n\left(1-\frac{1}{2n\alpha}\right)^{\alpha n\log{n}}
\leq c n e^{-\frac{\log{n}}{2}}= c\sqrt{n},
    \end{split}
\end{equation}
by Proposition \ref{VariProp}. Recall $d_{i,t}:=\textnormal{dist}(\sigma_t^{(i)},{\sigma'}_t^{(i)})=\frac{n}{2}|S_t^{(i)}-{S'}_t^{(i)}|$ and let us define the stopping time 
\[
\tau_0 := \min\{t\geq0: \max_{1\leq i\leq m}d_{i,t}\leq 1\}. 
\]
After time $t_n$, we will construct a coupling such that $(Y_t)_{t_n\leq t<\tau_0}$ with the stopping time $\tau_0$ that satisfies the condition of the Lemma \ref{Lem:SupMar}. After time $\tau_0$, we run the modified monotone coupling.

Then, let us consider the time $t\in[t_n,\tau_0)$. We partition the sites by $L_t:=\cup_{i:d_{i,t}\leq 1}G_i$ and $J_t:=\cup_{i:d_{i,t}> 1}G_i$. Note that $J_t\neq\emptyset$ for $t<\tau_0$. We choose a site uniformly over $V=L_t\dot\cup J_t$, where $\dot\cup$ indicates a disjoint union. If a site is chosen in the set $L_t$, we run the modified monotone coupling, otherwise we run two chains independently. This is a coupling of the Glauber dynamics. Note that once $G_i$ becomes a subset of $L_t$, it remains a subset of $L$. In addition, the sign of each $S_t^{(i)}-{S'}_t^{(i)}$ does not change if $t<\tau_0$.

Clearly, $Y_t$ has bounded increments, and from Proposition \ref{VariProp}, because $\beta<\beta_{cr}$ implies $\rho_n<1$, we have $(Y_t)_{t_n\leq t<\tau_0}$ as non-negative supermartingale. In addition, on $t\in[t_n,\tau_0)$, $J_t\neq\emptyset$; this implies $|J_t|\geq np_1$. If $I$, the uniform random variable over $V$, is chosen in the set $J_t$, $\sigma_{t+1}(I)$ and $\sigma_{t+1}'(I)$ are chosen independently. Hence, the conditional probability $\prob(Y_{t+1}\neq Y_t|\mathcal{F}_t)$ is bounded away from zero uniformly. Therefore, by Lemma \ref{Lem:SupMar}, for sufficiently large $\gamma n$,
\[
\prob_{\sigma,\sigma'}(\tau_0>t_n +\gamma n|\sigma_{t_n},\sigma_{t_n}')\leq \tilde{c}\frac{Y_{t_n}}{\sqrt{\gamma n}},
\]
for some positive constant $\tilde{c}$ not depending on $n$. Taking expectation coupled with (\ref{4.1}), we obtain
\[
\prob_{\sigma,\sigma'}(\tau_0>t_n +\gamma n)\leq O(\gamma^{-1/2}).
\]
Since $d_{i,\tau_0}$ is at most one for all $i$, in this degree $Y_{\tau_0}=\sum_{i=1}^{m}{a_i d_{i,\tau_0}}\leq \sum_{i=1}^{m}{a_i}=1$. As previously mentioned, from time $\tau_0$ onward, we run the modified monotone coupling to apply the modified version of Proposition \ref{VariProp} to obtain
\begin{align*}
    \prob_{\sigma,\sigma'}(\tau_{\text{mag}}>\tau_0+\gamma'n|\sigma_{\tau_0},\sigma'_{\tau_0})
    &\leq\prob_{\sigma,\sigma'}\left(Y_{\tau_0+\gamma'n}\geq a_{\text{min}}|\sigma_{\tau_0},\sigma'_{\tau_0}\right)\\
    &\leq \E_{\sigma,\sigma'}\left[Y_{\tau_0+\gamma'n}|\sigma_{\tau_0},\sigma'_{\tau_0}\right]/a_{\text{min}}\\
    &\leq \rho_n^{\gamma'n}\cdot a_\text{min}^{-1}
    = \left(1-\frac{1}{2n\alpha}\right)^{\gamma' n}\cdot a_\text{min}^{-1}
    \leq  e^{-\gamma'/(2\alpha )}\cdot a_\text{min}^{-1}.
\end{align*}
Accordingly,
\[
\prob_{\sigma,\sigma'}(\tau_{\text{mag}}>t_n+\gamma n+\gamma'n)\leq O(\gamma^{-1/2})+e^{-\gamma'/(2\alpha )}\cdot a_\text{min}^{-1}.
\]
We conclude that for sufficiently large $\gamma$,
\[
\prob_{\sigma,\sigma'}(\tau_{\text{mag}}>t_n+\gamma n)\leq O(\gamma^{-1/2}).
\]
\end{proof}

We will run the Glauber dynamics starting from the set of ``$\emph{good}$" configurations to establish cutoff. To begin with, we introduce a lemma from \cite[Lemma 3.3]{Curie-Ising-classic}.
\begin{lemma}
For any subset $\Omega_0\subseteq\Omega$ and stationary distribution $\mu$,
\begin{align*}
    d_n(t_0+t)&=\max_{\sigma\in\Omega}\Vert\prob_\sigma(\sigma_{t_0+t}\in\cdot)-\mu\Vert_{\textnormal{TV}}\\
    &\leq\max_{\sigma_0\in\Omega_0}\Vert\prob_{\sigma_0}(\sigma_{t}\in\cdot)-\mu\Vert_{\textnormal{TV}}+\max_{\sigma\in\Omega}\prob_\sigma(\sigma_{t_0}\notin\Omega_0).
\end{align*}
\end{lemma}
We will apply the aforementioned lemma to the set of ``$\emph{good}$" configurations that spins are assigned evenly and precisely defined by
\[
\Omega_0:=\{\sigma\in\Omega: |S^{(i)}(\sigma)|\leq p_i/2,\quad i=1,2,\cdots,m\}.
\]

For a configuration $\sigma\in\Omega$, we define the number of positive and negative spins in each set $G_i$ by
\[
u_\sigma^{(i)}:=|\{v\in G_i:\sigma^{(i)}(v)=1\}|,\quad v_\sigma^{(i)}:=|\{v\in G_i:\sigma^{(i)}(v)=-1\}|, 
\]
and let
\[
\mathcal{C}:=\{(u^{(1)},v^{(1)},\cdots,u^{(m)},v^{(m)}): ~\forall 1\leq i\leq m,~~~ u^{(i)}\wedge v^{(i)} \geq |G_i|/4\}.
\]
Note that $\sigma\in\Omega_0 \Longleftrightarrow (u_\sigma^{(1)},v_\sigma^{(1)},\cdots,u_\sigma^{(m)},v_\sigma^{(m)})\in\mathcal{C}$. By Proposition \ref{Prop:expected} (1), there exist $\zeta>0$ such that for all $\sigma\in\Omega$, $\max_{1\leq i\leq m}|\E_\sigma S_{\zeta n}^{(i)}|\leq p_1/4$. Therefore, by Proposition \ref{prop:VariBound} along with Chebyshev's inequality,
\begin{align*}
    \prob_{\sigma}(\sigma_{\zeta n}\notin\Omega_0)\leq\sum_{i=1}^m\prob_\sigma(|S_{\zeta n}^{(i)}|>p_i/2)
    &\leq\sum_{i=1}^m\prob_\sigma(|S_{\zeta n}^{(i)}-\E_\sigma S_{\zeta n}^{(i)}|>p_i/4)\\
    &\leq\frac{16}{p_1^2}\sum_{i=1}^m\Var_\sigma S_{\zeta n}^{(i)}=O(n^{-1}),
\end{align*}
to obtain
\begin{align}\label{GoodDistance}
    d_n(\zeta n + t)\leq \max_{\sigma_0\in\Omega_0}\Vert\prob_{\sigma_0}(\sigma_{t}\in\cdot)-\mu\Vert_{\textnormal{TV}}+O(n^{-1}).
\end{align}

\begin{definition} (2$m$-coordinate chain)
We fix a configuration $\bar{\sigma}\in\Omega_0$. For each $\sigma\in\Omega$ and $1\leq i\leq m$, we define
\begin{align*}
    U_i(\sigma)&:=|\{v\in G_i: \sigma^{(i)}(v)=\bar{\sigma}^{(i)}(v)=1\}|,\\
    V_i(\sigma)&:=|\{v\in G_i: \sigma^{(i)}(v)=\bar{\sigma}^{(i)}(v)=-1\}|.
\end{align*}
For the Glauber dynamics $(\sigma_t)_{t\geq 0}$ with starting configuration $\sigma_0\in\Omega$, we define the $2m$-coordinate process by
\begin{align}\label{eq:2m-coordinate}
{\bm{W}}_t:=(U_t^{(1)},V_t^{(1)},\cdots,U_t^{(m)},V_t^{(m)}):=(U_1(\sigma_t),V_1(\sigma_t),\cdots,U_m(\sigma_t),V_m(\sigma_t)).
\end{align}
\end{definition}
This $2m$-coordinate chain is a Markov chain in its state space $\mathcal{W}\subset \mathbb{N}^{2m}$ with transition probability depending on the fixed configuration $\bar{\sigma}$. We denote its stationary measure by $\nu$. We can also define the magnetization chain $(S_t^{(1)},\cdots,S_t^{(m)})$ by
\[
S_t^{(i)}:=\frac{\sum_{v\in G_i}\sigma_t^{(i)}(v)}{n}=\frac{2(U_t^{(i)}-V_t^{(i)})}{n}-\frac{\bar{u}^{(i)}-\bar{v}^{(i)}}{n}.
\]
According to symmetry, we borrow the following lemma from \cite[Lemma 4.4]{Hee}.
\begin{lemma}\label{TV}
When $(\sigma_t)_{t\geq 0}$ is the Glauber dynamics with starting configuration $\sigma$, and ${\bm{W}}_t$ is the corresponding $2m$-coordinate chain defined by \eqref{eq:2m-coordinate} starting at ${\bm{w}}$, then
\[
\Vert \prob_{\sigma}(\sigma_t\in\cdot)-\mu \Vert_{\textnormal{TV}}=\Vert \prob_{{\bm{w}}}((U_t^{(1)},V_t^{(1)},\cdots,U_t^{(m)},V_t^{(m)})\in\cdot)-\nu \Vert_{\textnormal{TV}}.
\]
\end{lemma}
In agreement with the aforementioned lemma, we will observe the total variance distance of the $2m$-coordinate chain instead of that of the original chain.

\begin{lemma}\label{lem:2mCoup}
Let us suppose two configuration $\sigma_0$ and $\sigma_0'$ satisfy $S^{(i)}(\sigma_0)=S^{(i)}(\sigma_0')$ for all $1\leq i\leq m$. For the fixed good configuration $\bar{\sigma}$, we define
\[
\mathcal{A}_i:=\left\{ \sigma\in\Omega :\min\{ U_i(\sigma),\bar{u_i}-U_i(\sigma), V_i(\sigma),\bar{v_i}-V_i(\sigma)\}\geq\frac{|G_i|}{16} \right\},\quad \mathcal{A}:=\bigcap_{i=1}^m\mathcal{A}_i.
\]
Then, there exists a coupling $(\sigma_t,\sigma_t')$ of the Glauber dynamics with starting configuration $(\sigma_0,\sigma_0')$ satisfying the following:

\begin{enumerate}
    \item $\bS_t=\bS_t'$ for all $t\geq 0$,
    \item For $R_t^{(i)}:=U_t'^{(i)}-U_t^{(i)}$, 
    \[
    \E_{\sigma_0,\sigma_0'}\left[ R_{t+1}^{(i)}-R_t^{(i)}|\sigma_t,\sigma_t' \right]=-\frac{R_t^{(i)}}{n}+O(n^{-2}),
    \]
    \item There exists a constant $c>0$ not depending on $n$ such that on the event $\{\sigma_t,\sigma_t'\in\mathcal{A}\}$, 
    $$\prob_{\sigma_0,\sigma'_0}\left( R_{t+1}^{(i)}-R_t^{(i)}\neq 0|\sigma_t,\sigma_t' \right)\geq c>0, \quad i=1,2,\cdots,m.$$ 
\end{enumerate}
\end{lemma}
\begin{proof}
We construct the coupling as follows: $I$ and $U$ are uniform random variables over $V$ and $[0,1]$ respectively. We then generate a random spin $S$ from $I$ and $U$ by
\[
S:=\sum_{i=1}^m\mathbbm{1}_{I\in G_i}\left(\mathbbm{1}_{U\leq r_+\left(\sum_{1\leq j\leq m}{K_{ij}M_j}-K_{ii}\sigma(I)\right)}-\mathbbm{1}_{U>r_+\left(\sum_{1\leq j\leq m}{K_{ij}M_j}-K_{ii}\sigma(I)\right)}\right),
\]
and set
\[
\sigma_{t+1}(v)=
\begin{cases}
\sigma_t(v),& v\neq I,\\
S,& v=I.
\end{cases}
\]
Now for $\sigma_t'$, if $I\in G_i$ for some $i$, we choose $I'$ uniformly from $\{v\in G_i': \sigma_t'(v)=\sigma_t(I)\}$, and set
\[
\sigma'_{t+1}(v)=
\begin{cases}
\sigma'_t(v),& v\neq I',\\
S,& v=I'.
\end{cases}
\]
By definition, the magnetizations keep the agreement that $\bS_t=\bS'_t$ holds for all $t\geq 0$ and $(\sigma'_t)$ satisfies the Glauber dynamics.

For the given configuration $\sigma$, we divide the sites of each set $G_i$ ($i=1,2,\cdots,m$) into four sets (Fig \ref{tablebutfigure}),
\begin{align*}
    A_i(\sigma)&:=\{v\in G_i:\sigma(v)=\bar{\sigma}(v)=1\},\\
    B_i(\sigma)&:=\{v\in G_i:\sigma(v)=-1, \bar{\sigma}(v)=1\},\\
    C_i(\sigma)&:=\{v\in G_i:\sigma(v)=1, \bar{\sigma}(v)=-1\},\\
    D_i(\sigma)&:=\{v\in G_i:\sigma(v)=\bar{\sigma}(v)=-1\}.
\end{align*}
Note that
\begin{align*}
|A_i(\sigma)|=U_i(\sigma), ~|B_i(\sigma)|=\bar{u}_i-U_i(\sigma), ~|C_i(\sigma)|=\bar{v}_i-V_i(\sigma), ~|D_i(\sigma)|=V_i(\sigma)|.
\end{align*}

\begin{table}[h]
\begin{center}
\begin{tabular}{|l|cccccc|cccccc|}
\hline
$\bar{\sigma}^{(i)}$&+&+              &+                     &+& + & + & $-$ & $-$ & $-$ & $-$ & $-$ & $-$ \\ 
                    & &$\bar{u}_i$    &\multicolumn{1}{l}{ }& &               & & & $\bar{v}_i$& &\multicolumn{1}{l}{}& & \\ \hline
$\sigma_t^{(i)}$    &+&+              &\multicolumn{1}{l|}{+}&$-$& $-$ & $-$ & + & + & + & \multicolumn{1}{l|}{+} & $-$ & $-$ \\
                    & &$A_i(\sigma_t)$&\multicolumn{1}{l|}{ }& &$B_i(\sigma_t)$& & & $C_i(\sigma_t)$& &\multicolumn{1}{l|}{}&$D_i(\sigma_t)$& \\ \hline
$\sigma_t'^{(i)}$   &+&+              &+             &\multicolumn{1}{l|}{+}& $-$ & $-$ & + & + & \multicolumn{1}{l|}{+} & $-$ & $-$ & $-$ \\
                    & &$A_i(\sigma'_t)$&             &\multicolumn{1}{l|}{} &$B_i(\sigma'_t)$& & &$C_i(\sigma'_t)$&\multicolumn{1}{l|}{}& &$D_i(\sigma'_t)$&   \\ \hline
\end{tabular}
\end{center}
\captionof{figure}{Partitions of the vertices of $\sigma_t^{(i)}$ and $\sigma_t'^{(i)}$} \label{tablebutfigure}
\end{table}

The possible values of $R_{t+1}^{(i)}-R_{t}^{(i)}$ follow Table \ref{table:del R}.

\begin{table}[h]
\setcounter{table}{0}
\begin{center}
\begin{tabular}{llr|c}
$I$                  & $I'$                   & $S$ & $R^{(i)}_{t+1}-R^{(i)}_t$ \\\hline
$I\in B_i(\sigma_t)$ & $I'\in D_i(\sigma'_t)$ & +1  & --1        \\
$I\in C_i(\sigma_t)$ & $I'\in A_i(\sigma'_t)$ & --1  & --1        \\
$I\in A_i(\sigma_t)$ & $I'\in C_i(\sigma'_t)$ & --1  & +1          \\
$I\in D_i(\sigma_t)$ & $I'\in B_i(\sigma'_t)$ & +1  & +1        \\
\multicolumn{3}{l|}{otherwise} & 0                                       
\end{tabular}
\end{center}
\caption{$R_{t+1}^{(i)}-R_{t}^{(i)}$}
\label{table:del R}
\end{table}

Considering $S_t^{(i)}={S'}_t^{(i)}$, we obtain $R_t^{(i)}=U_t'^{(i)}-U_t^{(i)}=V_t'^{(i)}-V_t^{(i)}$. Hence,
\begin{align*}
    \prob_{\sigma_0,\sigma'_0}(R_{t+1}^{(i)}-R_{t}^{(i)}=-1|\sigma_t,\sigma_t')=:a(U_t^{(i)},V_t^{(i)},{U'}_t^{(i)},{V'}_t^{(i)}),\\
    \prob_{\sigma_0,\sigma'_0}(R_{t+1}^{(i)}-R_{t}^{(i)}=+1|\sigma_t,\sigma_t')=:b(U_t^{(i)},V_t^{(i)},{U'}_t^{(i)},{V'}_t^{(i)}),
\end{align*}
where
\begin{align*}
    a(U_t^{(i)},V_t^{(i)},{U'}_t^{(i)},{V'}_t^{(i)})
    &=\frac{\bar{u}_i-U_t^{(i)}}{n}\frac{V_t'^{(i)}}{\bar{u}_i-{U'}_t^{(i)}+{V'}_t^{(i)}} r_+\left(\sum_{1\leq j\leq m}{K_{ij}M_j}+K_{ii}\right)\\
    &+\frac{\bar{v}_i-V_t^{(i)}}{n}\frac{U_t'^{(i)}}{\bar{v}_i-{V'}_t^{(i)}+{U'}_t^{(i)}} r_-\left(\sum_{1\leq j\leq m}{K_{ij}M_j}-K_{ii}\right),\\
    b(U_t^{(i)},V_t^{(i)},{U'}_t^{(i)},{V'}_t^{(i)})
    &=\frac{U_t^{(i)}}{n}\frac{\bar{v}_i-{V'}_t^{(i)}}{{U'}_t^{(i)}+\bar{v}_i-{V'}_t^{(i)}} r_-\left(\sum_{1\leq j\leq m}{K_{ij}M_j}-K_{ii}\right)\\
    &+\frac{V_t^{(i)}}{n}\frac{\bar{u}_i-{U'}_t^{(i)}}{\bar{u}_i-{U'}_t^{(i)}+{V'}_t^{(i)}} r_+\left(\sum_{1\leq j\leq m}{K_{ij}M_j}+K_{ii}\right).
\end{align*}
Therefore, using $R_t^{(i)}=U_t'^{(i)}-U_t^{(i)}=V_t'^{(i)}-V_t^{(i)}$, we obtain
\begin{align*}
    \E_{\sigma_0,\sigma_0'}\bigg[R_{t+1}^{(i)}-R_t^{(i)}|\sigma_t,&\sigma_t'  \bigg]= b(U_t^{(i)},V_t^{(i)},{U'}_t^{(i)},{V'}_t^{(i)})-a(U_t^{(i)},V_t^{(i)},{U'}_t^{(i)},{V'}_t^{(i)})\\
    &=-\frac{R_t^{(i)}}{n}\left( r_+\left(\sum_{1\leq j\leq m}{K_{ij}M_j}+K_{ii}\right)+r_-\left(\sum_{1\leq j\leq m}{K_{ij}M_j}-K_{ii}\right) \right)\\[-2pt]
    &=-\frac{R_t^{(i)}}{n}+O(n^{-2}),
\end{align*}
where the last equality follows from \eqref{control_r}. In addition, on the event $\{\sigma_t,\sigma_t'\in\mathcal{A} \}$,
\[
\prob_{\sigma_0,\sigma'_0}\left(R_{t+1}^{(i)}-R_{t}^{(i)}\neq 0\big |\sigma_t,\sigma_t'\right)\geq b(U_t^{(i)},V_t^{(i)},{U'}_t^{(i)},{V'}_t^{(i)})\geq c,
\]
for some constant $c>0$, independent of $n$, because $r_+$ and $r_-$ are uniformly bounded away from zero and one.
\end{proof}

Now, we demonstrate Theorem \ref{thm:cutoffUpper}.
\begin{proof}[Proof of Theorem \ref{thm:cutoffUpper}]
Combining Lemma \ref{TV} and (\ref{GoodDistance}), we have
\[
d_n(t+\zeta n)\leq \max_{{\bm{w}}\in\mathcal{C}}\Vert \prob_{{\bm{w}}}({\bm{W}}_t\in\cdot)-\nu \Vert_{\textnormal{TV}}+O(n^{-1}).
\]
For given $2m$-coordinate chains ${\bm{W}}_t$ and ${\bm{W}}'_t$ starting from ${\bm{w}}$ and ${\bm{w}}'$, we define the stopping times when the $i$th coordinates and the full chain coincide by
\[
\tau_{i,c}:=\min\{t\geq 0:(U_t^{(i)},V_t^{(i)})=({U'}_t^{(i)},{V'}_t^{(i)})\},\quad \tau_c:=\min\{t\geq 0:{\bm{W}}_t={\bm{W}}'_t\}.
\]
It is well-known that (e.g \cite[Section 5.2]{MCMC}),
\[
\Vert \prob_{{\bm{w}}}({\bm{W}}_t\in\cdot)-\prob_{{\bm{w}}'}({\bm{W}}'_t\in\cdot) \Vert_{\textnormal{TV}}\leq \prob_{{\bm{w}},{\bm{w}}'}(\tau_c>t),
\]
follows; thus we obtain
\[
\max_{{\bm{w}}\in\mathcal{C}}\Vert \prob_{{\bm{w}}}({\bm{W}}_t\in\cdot)-\nu \Vert_{\textnormal{TV}}
\leq \max_{\substack{{\bm{w}}\in\mathcal{C} \\{\bm{w}}'\in\mathcal{W}}} \prob_{{\bm{w}},{\bm{w}}'}(\tau_c>t).
\]
Hence, all we have to do is to bound the right-hand side. We define $t_n(\gamma):=t_n+\gamma n$ and fix a good starting figuration $\bar{\sigma}\in\Omega_0$ with the corresponding $2m$-coordinate chain $\mathbf{\bar{{\bm{w}}}}=(\bar{u}_1,\bar{v}_1,\cdots,\bar{u}_m,\bar{v}_m)\in\mathcal{C}$. We will construct a suitable coupling by two steps.

The first step is the magnetization coupling phase. By Lemma \ref{lem:MagCoup}, there exists a coupling $(\sigma_t,\sigma_t')$ with starting configuration $(\bar{\sigma}, \sigma')$ such that for $H_M:=\{\tau_{\text{mag}}\leq t_n(\gamma)\}$, $\prob(H_M^c)\leq O(\gamma^{-1/2})$ holds.

The next step is the $2m$-coordinate chain coupling phase. On the event $H_M$, we run the coupling defined on Lemma \ref{lem:2mCoup} for $t>t_n(\gamma)$, and on the event $H_M^c$, which is unlikely to occur for large $\gamma$, we run two chains independently for $t>t_n(\gamma)$. To deal with the situation described on Lemma \ref{lem:2mCoup}, we define
\[
H_i(t):=\{\sigma_t^{(i)}\in\mathcal{A}_i,{\sigma'}_t^{(i)}\in\mathcal{A}_i\},\quad H_i:=\bigcap_{t_n(\gamma)\leq t\leq t_n(2\gamma)} H_i(t),\quad H_{G}:=\bigcap_{i=1}^m H_i,
\]
and we claim
\[
\prob_{\bar{\sigma},\sigma'}(H_i^c)\leq \gamma O(n^{-1}),\quad i=1,2,\cdots,m.
\]
To that end, let us consider the set
\[
\mathcal{G}_i:=\{v\in G_i : \bar{\sigma}(v)=+1\},\quad i=1,2,\cdots,m.
\]
Note that
\begin{align*}
\{\sigma_t^{(i)}\notin\mathcal{A}_i\}\subset \{U_t^{(i)}<np_i/16\}&\cup\{\bar{u}_i-U_t^{(i)}<np_i/16\}\\
&\cup\{V_t^{(i)}<np_i/16\}\cup\{\bar{v}_i-V_t^{(i)}<np_i/16\},
\end{align*}
and because $\bar{\sigma}\in\Omega_0$, we have $\bar{u}_i\geq np_i/4$; thus, $U_t^{(i)}< np_i/16$ implies $\bar{u}_i-U_t^{(i)}\geq 3np_i/16$. Again, this implies that $|M_t(\mathcal{G}_i)|=\frac{1}{2}|\sum_{v\in \mathcal{G}_i}\sigma_t(v)|=\frac{1}{2}|U_t^{(i)}-(\bar{u}_i-U_t^{(i)})|\geq np_i/16$. Similarly, we can obtain that $\bar{u}_i-U_t^{(i)}<np_i/16$ implies $|M_t(G_i/\mathcal{G}_i)|\geq np_i/16$. Now, let us define
\[
B^*:=\bigcup_{t_n(\gamma)\leq t\leq t_n(2\gamma)} \Big\{|M_t(\mathcal{G}_i)|\geq np_i/16\Big\},\quad Y^*:=\sum_{t_n(\gamma)\leq t\leq t_n(2\gamma)} \mathbbm{1}_{|M_t(\mathcal{G}_i)|\geq np_i/32}.
\]
Since $M_t(\mathcal{G}_i)$ has increments of 0, 1, or $-1$, we have $B^*\subseteq\{Y^*\geq np_i/32\}$. Accordingly, with Markov's inequality, we obtain
\[
\prob_{\bar{\sigma},\sigma'}(B^*)\leq\prob_{\bar{\sigma},\sigma'}(Y^*\geq np_i/32)\leq32\E_{\bar{\sigma},\sigma'}[Y^*]/(np_i).
\]
Observing that $\rho_n^t=O(\sqrt{n})$ holds for $t\geq t_n$, by combining Proposition \ref{Prop:expected} (2) and Chebyshev's inequality, we obtain
$\prob_{\bar{\sigma},\sigma'}\left(|M_t(\mathcal{G}_i)|\geq np_i/32\right)=O(n^{-1})$, so $\E_{\bar{\sigma},\sigma'}[Y^*]=\gamma O(1)$ to get $\prob_{\bar{\sigma},\sigma'}(B^*)=\gamma O(n^{-1})$. Similar results hold for $\{M_t(G_i/\mathcal{G}_i)\geq np_i/16\}$ and for $M'$ instead of $M$ as well; this verifies our claim,
\[
\prob_{\bar{\sigma},\sigma'}(H_i^c)\leq 8\gamma O(n^{-1})=\gamma O(n^{-1}),\quad i=1,2,\cdots,m.
\]
Therefore, $\prob_{\bar{\sigma},\sigma'}(H_G^c)\leq \sum_{i=1}^m\prob_{\bar{\sigma},\sigma'}(H_i^c) \leq\gamma O(n^{-1})$. 

Our last claim is
\[
\prob_{\bar{\sigma},\sigma'}(\tau_{i,c}>t_n(2\gamma)\cap H_i \cap H_M)=O(\gamma^{-1/2}).
\]
On purpose, according to Lemma \ref{Lem:SupMar} and Lemma \ref{lem:2mCoup}, since $R_t^{(i)}$ does not reach $0$ while $t<\tau_{i,c}$, we have
\[
\prob_{\bar{\sigma},\sigma'}\left(\tau_{i,c}>t_n(2\gamma)\cap H_i \cap H_M\big |\sigma_{t_n(\gamma)},\sigma_{t_n(\gamma)}'\right)\leq\frac{c|R_{t_n(\gamma)}^{(i)}|}{\sqrt{n \gamma}},
\]
for some constant $c$, and taking expectation, we obtain
\[
\prob_{\bar{\sigma},\sigma'}\left(\tau_{i,c}>t_n(2\gamma)\cap H_i \cap H_M\right)\leq\frac{c\E_{\bar{\sigma},\sigma'}|R_{t_n(\gamma)}^{(i)}|}{\sqrt{n \gamma}},
\]
Meanwhile, $\E_{\bar{\sigma},\sigma'}|R_{t_n(\gamma)}^{(i)}|\leq \E_{\bar{\sigma}}|M_{t_n(\gamma)}(\mathcal{G}_i)|+\E_{\sigma'}|M'_{t_n(\gamma)}(\mathcal{G}_i)|=O(\sqrt{n})$ follows by Proposition \ref{Prop:expected} because $|R_t^{(i)}|=|{U'}_t^{(i)}-U_t^{(i)}|=|M_t'(\mathcal{G}_i)-M_t(\mathcal{G}_i)|$, to verify the claim. In conclusion,
\begin{align*}
    \prob_{\bar{\sigma},\sigma'}(\tau_c>t_n(2\gamma))
    &\leq \prob_{\bar{\sigma},\sigma'}(\tau_c>t_n(2\gamma)\cap H_M\cap H_G)+\prob_{\bar{\sigma},\sigma'}(H_M^c)+\prob_{\bar{\sigma},\sigma'}(H_G^c)\\
    &\leq \sum_{i=1}^m\prob_{\bar{\sigma},\sigma'}(\tau_{i,c}>t_n(2\gamma)\cap H_M\cap H_G)+\prob_{\bar{\sigma},\sigma'}(H_M^c)+\prob_{\bar{\sigma},\sigma'}(H_G^c)\\
    &=O(\gamma^{-1/2})+O(\gamma^{-1/2})+\gamma O(n^{-1})=O(\gamma^{-1/2})+\gamma O(n^{-1}),
\end{align*}
to yield
\[
d_n(t_n+(2\gamma+\zeta)n)\leq O(\gamma^{-1/2})+\gamma O(n^{-1})+O(n^{-1}).
\]
By taking $n\rightarrow \infty$ and then $\gamma \rightarrow\infty$, the limit goes to zero.
\end{proof}

\subsection{Lower Bound}
In order to verify the lower bound in the high temperature regime, we use the following lemma from \cite[Proposition 7.9]{MCMC}.
\begin{lemma}\label{lem:lower}
Let $f:\mathcal{S}\rightarrow \mathbb{R}$ be a measurable function, and $\nu_1,\nu_2$ be two probability measures on $\mathcal{S}$. We define $\sigma_*^2:=\max\{\Var_{\nu_1}f,\Var_{\nu_2}f\}$. If $|\E_{\nu_1}f-\E_{\nu_2}f|\geq r \sigma_*$, then
\[
\Vert\nu_1-\nu_2\Vert_{\textnormal{TV}}\geq 1-\frac{8}{r^2}.
\]
\end{lemma}

We now restate the lower bound part of Theorem \ref{Thm:HighTemp}.
\begin{theorem} When $\beta < \beta_{cr}$, 
\[
\lim_{\gamma\rightarrow\infty}\lim_{n\rightarrow\infty} d_n(t_n-\gamma n) =1.
\]
\end{theorem}
\begin{proof}
Recall the drift of $\bS_t$ from \eqref{eq:drift} that
\begin{align*}
    \mathbb{E}\left[S_{t+1}^{(i)}-S_t^{(i)}|\mathcal{F}_t\right]
    =\frac{1}{n}\left( -S_t^{(i)} +f_n({\bS}_t)+g_n({\bS}_t) \right),
\end{align*}
and $X_t^{(i)}=\sum_{j=1}^m{K_{ij}M_t^{(j)}}$ with
\begin{align*}
    f_i({\bS}_t)=p_i\left(\tanh{\beta X_t^{(i)}}+O(n^{-1})\right),\quad
    g_i({\bS}_t)=-S_t^{(i)}\left(1+O(n^{-1})\right).
\end{align*}
Note that $\tanh{x}\geq x-\frac{x^2}{3}$ implies $f_n({\bS}_t)\geq p_i\beta X_t^{(i)}-\frac{p_i\beta^2}{3}(X_t^{(i)})^2+O(n^{-1})$. Hence, 
\begin{align*}
    \mathbb{E}\left[S_{t+1}^{(i)}-S_t^{(i)}|\mathcal{F}_t\right]
    \geq \frac{1}{n}\left( -S_t^{(i)}+ p_i\beta X_t^{(i)}-\frac{p_i\beta^2 (X_t^{(i)})^2}{3} \right)+O\left(\frac{1}{n^2}\right),
\end{align*}
to rewrite,
\begin{align*}
    \mathbb{E}\left[S_{t+1}^{(i)}|\mathcal{F}_t\right]
    &\geq \left(1-\frac{1}{n}\right)S_t^{(i)}+ \frac{p_i\beta X_t^{(i)}}{n}-\frac{p_i \beta^2(X_t^{(i)})^2}{3n} +O\left(\frac{1}{n^2}\right).
\end{align*}
Equivalently, we obtain
\begin{align*}
    \mathbb{E}\left[{\bS}_{t+1}|\mathcal{F}_t\right]
    &\geq \mathbf{Q}{\bS}_t-{\bm{x}}+O(n^{-2}),
\end{align*}
where ${\bm{x}}=\frac{\beta^2}{3n}\Big(p_1 (X_t^{(1)})^2,\cdots,p_m (X_t^{(m)})^2\Big)^{\top}$. Therefore,
\begin{align*}
    \mathbb{E}\left[\ba^{\top}{\bS}_{t+1}|\mathcal{F}_t\right]
    &\geq \rho_n \ba^{\top}{\bS}_t-\ba^{\top}{\bm{x}}+O(n^{-2}).
\end{align*}
Note that $\ba^{\top}{\bm{x}}=\frac{\beta^2}{3n}\sum_{i=1}^m{a_i p_i (X_t^{(i)})^2}$. Let $k_M:=\max_{i,j}{k_{ij}}$. Then,
\begin{align*}
    \sum_{i=1}^m{a_i p_i (X_t^{(i)})^2}\leq \sum_{i=1}^m{a_i p_i \left(\sum_{j=1}^m k_M|S_t^{(j)}|\right)^2}= k_M^2\left(\sum_{i=1}^m a_ip_i\right) \Vert \bS_t \Vert_1^2+O(n^{-2}).
\end{align*}
Hence,
\[
\mathbb{E}\left[\ba^{\top}{\bS}_{t+1}|\mathcal{F}_t\right]
    \geq \rho_n \ba^{\top}{\bS}_t- \frac{\beta^2 k_M^2\sum_{i=1}^m{a_ip_i}}{3n}\Vert\bS_t\Vert_1^2 + O(n^{-2}).
\]
By taking expectation, we get
\[
\mathbb{E}\left[\ba^{\top}{\bS}_{t+1}\right]
    \geq \rho_n \mathbb{E}\left[\ba^{\top}{\bS}_t\right]-\frac{\beta^2 k_M^2\sum_{i=1}^m{a_ip_i}}{3n}\E\Vert\bS_t\Vert_1^2 + O(n^{-2}).
\]
By Proposition \ref{prop:VariBound} and Proposition \ref{Prop:expected}-(3),
\begin{align*}
    \E\Vert\bS_t\Vert_1^2=(\E\Vert\bS_t\Vert_1)^2&+\Var\Vert\bS_t\Vert_1=(\E\Vert\bS_t\Vert_1)^2+O(n^{-1})\\
    &\leq \rho_n^{2t}\sum_{i=1}^m\frac{(s^{(i)})^2}{p_i}+2\rho_n^t\left(\sum_{i=1}^m\frac{(s^{(i)})^2}{p_i}\right)^{1/2}O(n^{-1/2})+O(n^{-1}).
\end{align*}
to get
\begin{align*}
\mathbb{E}\bigg[\ba^{\top}{\bS}_{t+1}& \bigg]\geq \mathbb{E}\left[\ba^{\top}{\bS}_t\right]\\
&-\frac{\beta^2 k_M^2\sum_{i=1}^m{a_ip_i}}{3n}\left(\rho_n^{2t}\sum_{i=1}^m\frac{(s^{(i)})^2}{p_i}+2\rho_n^t\left(\sum_{i=1}^m\frac{(s^{(i)})^2}{p_i}\right)^{1/2}O(n^{-1/2})\right) + O(n^{-2}).
\end{align*}
Let us define $Z_t:=\displaystyle\frac{\ba^{\top}{\bS}_t}{\rho_n^t}$ to obtain
\begin{align*}
    \mathbb{E}&Z_{t+1}-\mathbb{E}Z_t\\
    &\geq -\frac{\beta^2 k_M^2 \sum_{i=1}^m{a_ip_i}}{3n \rho_n^{t+1}}
    \left(\rho_n^{2t}\sum_{i=1}^m\frac{(s^{(i)})^2}{p_i}+2\rho_n^t\left(\sum_{i=1}^m\frac{(s^{(i)})^2}{p_i}\right)^{1/2}O(n^{-1/2})\right) + \frac{1}{\rho_n^{t+1}}O(n^{-2})\\
    &= -\frac{\beta^2 k_M^2 \sum_{i=1}^m{a_ip_i}}{3(n-1/(2\alpha))}
    \left(\rho_n^{t}\sum_{i=1}^m\frac{(s^{(i)})^2}{p_i}+2\left(\sum_{i=1}^m\frac{(s^{(i)})^2}{p_i}\right)^{1/2}O(n^{-1/2})\right) + \frac{1}{\rho_n^{t+1}}O(n^{-2}).
\end{align*}
Summing up from $0$ to $t-1$, we have
\begin{align*}
    \mathbb{E}Z_{t}-\mathbb{E}Z_0
    &\geq -\frac{\beta^2 k_M^2 \sum_{i=1}^m{a_ip_i}}{3(n-1/(2\alpha))}\left(\sum_{i=1}^m\frac{(s^{(i)})^2}{p_i}\frac{1-\rho_n^t}{1-\rho_n}+t\cdot O(n^{-1/2})\right) +\frac{\rho_n^{-t}-1}{1-\rho_n}O(n^{-2}).
\end{align*}
Consider the time $t_*:=t_n-2\gamma\alpha n$. Note that $\rho_n^{t_*}\geq \frac{e^\gamma}{n^{n/(2(n-(2\alpha)^{-1}))}}$ to obtain
\begin{align*}
    \mathbb{E}Z_{t_*}- & \sum_{i=1}^m  {a_is_i}
    \geq -\frac{\beta^2 k_M^2 \sum_{i=1}^m{a_ip_i}}{3(1-(2\alpha n)^{-1})/(2\alpha)}\sum_{i=1}^m\frac{(s^{(i)})^2}{p_i}(1-\rho_n^{t_*})\\
    &-\frac{\beta^2 k_M^2 \sum_{i=1}^m{a_ip_i}}{3(n-(2\alpha)^{-1})}t_*\cdot O\left(\frac{1}{n^2}\right)+\frac{\rho_n^{-t_*}-1}{1-g}O\left(\frac{1}{n^2}\right)\\
    &\geq -\frac{\beta^2k_M^2\sum_{i=1}^m{a_ip_i}}{3(1-(2\alpha n)^{-1})/(2\alpha)}\sum_{i=1}^m\frac{(s^{(i)})^2}{p_i}\left(1-\frac{e^\gamma}{n^{n/2(n-(2\alpha)^{-1})}}\right)\\
    &-\frac{\alpha\beta^2 k_M^2 \sum_{i=1}^m{a_ip_i}}{3(n-(2\alpha)^{-1})}\left(n\log n-{2\gamma n}\right) O\left(\frac{1}{n^2}\right)+2\alpha\left(\frac{n^{n/(2(n-(2\alpha)^{-1}))}}{e^\gamma}-1\right)O\left(\frac{1}{n}\right).
\end{align*}
For every $\gamma$, the right-hand side converges to $-\frac{2\alpha\beta^2k_M^2\sum_{i=1}^m{a_ip_i}}{3}\sum_{i=1}^m\frac{(s^{(i)})^2}{p_i}$ as $n\rightarrow\infty$. Let us define $c:=\frac{2\alpha\beta^2k_M^2\sum_{i=1}^m{a_ip_i}}{3}$. We have $\sum_{i=1}^m a_is_i-c\sum_{i=1}^m \frac{(s^{(i)})^2}{p_i}>0$ for any $0<s^{(i)}<\frac{a_ip_i}{c}$. Let us fix one of the satisfying starting states $\bs$ and, for that value, there exists a positive constant $\epsilon$ such that $\E Z_{t_*}>\epsilon$. Consequently, for sufficiently large $n$, we have
\[
\E_\bs(\ba^{\top}\bS_{t_*})> \epsilon \rho_n^{t_*}\geq \epsilon\frac{e^\gamma}{n^{n/(2(n-(2\alpha)^{-1}))}}\geq \frac{\epsilon e^\gamma}{\sqrt{n}}.
\]
Since $\Var({\ba}^{\top}{\bS}_{t_*})=O(n^{-1})$ by Proposition \ref{prop:VariBound}, combining with Lemma \ref{lem:lower}, there exists some positive constant $c'$ such that
\[
\lim_{\gamma\rightarrow\infty}\liminf_{n\rightarrow\infty} d_n\left(t_n-{2\alpha\gamma n}\right)\geq \lim_{\gamma\rightarrow\infty} 1-\frac{c'}{\epsilon^2 e^{2\gamma}}=1.
\]
\end{proof}

\section{Mixing time at Critical Temperature}\label{sec:critical}
In this section, we prove $O(n^{3/2})$ mixing time at the critical temperature, i.e., $\beta=\beta_{cr}$. Recall the matrices $\mathbf{K},\mathbf{P},\mathbf{D}, \mathbf{B}$ from \eqref{eq:def_matrix}. In addition, recall from Proposition \ref{prop:eigen} that $\beta=\beta_{cr}$ implies that the largest eigenvalue of the matrix $\beta\mathbf{B}$ is equal to 1. Since the matrix $\mathbf{B}=\mathbf{D}^2\mathbf{K}$ is similar to the symmetric matrix $\mathbf{D}\mathbf{K}\mathbf{D}$, the largest eigenvalue of $\beta\mathbf{D}\mathbf{K}\mathbf{D}$ is 1 as well. Since $\mathbf{K}$ is positive definite, spectral decomposition implies there exists orthogonal matrix $\mathbf{V}$ and diagonal matrix $\Lambda:=\text{diag}(\lambda_1,\lambda_2,\cdots,\lambda_{m-1},1)$ such that
\begin{align*}
    \beta \mathbf{D}\mathbf{K}\mathbf{D}=\mathbf{V}\text{diag}(\lambda_1,\cdots,\lambda_{m-1},1)\mathbf{V}^\top=\mathbf{V}\Lambda\mathbf{V}^\top,
\end{align*}
with the eigenvalues $0<\lambda_1\leq \lambda_2\leq\cdots\lambda_{m-1}<\lambda_m =1$.

\subsection{Upper Bound}\label{Thm:5.1}
\begin{theorem} When $\beta = \beta_{cr}$, we have $t_{\text{mix}}=O(n^{3/2})$.
\end{theorem}

Note that it takes $O(n\log n)$ steps for two configurations to coincide after their magnetizations agree by Lemma \ref{Lem:CoupleSameMag}. Therefore, it is sufficient to show that the magnetization coincides in $O(n^{3/2})$ steps. We first show that the magnetization on each group almost goes to zero in $O(n^{3/2})$ steps.
\begin{lemma}
We define $\tau_0 :=\min\{t\geq 0:\max_{1\leq i\leq m}\{|S_t^{(i)}|\}\leq 1/n \}$. Then there exist positive constant $c_1$ such that
\[
\prob_\sigma(\tau_0>(c_1+\gamma)n^{3/2})=O(\gamma^{-1/2}).
\]
\end{lemma}
\begin{proof}
From now on, we discuss on the event $\{t<\tau_0\}$. Recall the drift of $\bS_t$ from \eqref{eq:drift} that
\[
\E[S_{t+1}^{(i)}|\mathcal{F}_t]=\left(1-\frac{1}{n}\right)S_t^{(i)}+\frac{p_i}{n}\tanh(\beta X_t^{(i)})+O(\frac{1}{n^2}).
\]
Recall $X_t^{(i)}$ from \eqref{eq:def_X} to obtain $\Var[X_{t+1}^{(i)}|\mathcal{F}_t]\leq\E[(X_{t+1}^{(i)}-X_{t}^{(i)})^2|\mathcal{F}_t]=O(n^{-2})$ and
\begin{align*}
    \E[X_{t+1}^{(i)}|\mathcal{F}_t]&=\left(1-\frac{1}{n}\right)X_t^{(i)}+\frac{1}{n}\sum_{j=1}^m k_{ij}p_j\tanh(\beta X_t^{(j)}) +O(n^{-2}), \nonumber\\
    \E[(X_{t+1}^{(i)})^2|\mathcal{F}_t]
    &=(\E[X_{t+1}^{(i)}|\mathcal{F}_t])^2+\Var[X_{t+1}^{(i)}|\mathcal{F}_t] \nonumber\\
    &=\left(1-\frac{2}{n}\right)(X_t^{(i)})^2+\frac{2}{n}X_t^{(i)}\sum_{j=1}^m k_{ij}p_j\tanh(\beta X_t^{(j)}) +O(n^{-2}). \nonumber
\end{align*}
We compute the drift of $\mathbf{X}_{t}^{\top} {\mathbf{P}}\mathbf{X}_{t}=\sum_{i=1}^m p_i \left(X_t^{(i)}\right)^2$.
\begin{align}\label{eq:eq2}
    &\E[\mathbf{X}_{t+1} \mathbf{P} \mathbf{X}_{t+1}|\mathcal{F}_t]=\left(1-\frac{2}{n}\right)\mathbf{X}_{t}^{\top} \mathbf{P} \mathbf{X}_{t}+\frac{2}{n}\sum_{i=1}^m p_iX_t^{(i)}\sum_{j=1}^m k_{ij}p_j\tanh(\beta X_t^{(j)}) +O(n^{-2}),\nonumber\\
    &\E[\mathbf{X}_{t+1}^{\top} \mathbf{P} \mathbf{X}_{t+1}-\mathbf{X}_{t}^{\top} \mathbf{P} \mathbf{X}_{t}|\mathcal{F}_t]=-\frac{2}{n}\Big(\mathbf{X}_{t}^{\top} \mathbf{P} \mathbf{X}_{t}-\sum_{i=1}^m p_iX_t^{(i)}\sum_{j=1}^m k_{ij}p_j\tanh(\beta X_t^{(j)})\Big) +O(n^{-2}).
\end{align}
We take conditional expectation with respect to $\mathcal{F}_{t-1}$ on both sides of \eqref{eq:eq2}, and use tower property of conditional expectations to obtain
\begin{equation}\label{C0}
\begin{split}
    \E[\mathbf{X}_{t+1}^{\top} \mathbf{P} &\mathbf{X}_{t+1}-\mathbf{X}_{t}^{\top} \mathbf{P} \mathbf{X}_{t}|\mathcal{F}_{t-1}]\\
    &=-\frac{2}{n}\E[\mathbf{X}_{t}^{\top} \mathbf{P} \mathbf{X}_{t}|\mathcal{F}_{t-1}]+\frac{2}{n}\E\left[\sum_{i=1}^m p_iX_t^{(i)}\sum_{j=1}^m k_{ij}p_j\tanh(\beta X_t^{(j)})\Big|\mathcal{F}_{t-1}\right] +O(n^{-2}).
\end{split}
\end{equation}
Meanwhile, since the shift of $X^{(i)}$, i.e., $X_{t+1}^{(i)}-X_t^{(i)}$ is $O(n^{-1})$, the conditional variance of both $\Var(|X_t^{(i)}||\mathcal{F}_{t-1})$ and $\Var\left(\E[|X_t^{(i)}||\mathcal{F}_{t-1}]\right)$ are $O(n^{-2})$. Therefore, we have
\begin{align*}
\E[\bX_{t}^{\top} \mathbf{P}\bX_{t}|\mathcal{F}_{t-1}]
&=\E\left[\sum_{i=1}^m p_i (X_t^{(i)})^2\Big|\mathcal{F}_{t-1}\right]
=\sum_{i=1}^m p_i\left(\E[|X_t^{(i)}||\mathcal{F}_{t-1}]\right)^2+O(n^{-2}),\\
\E\left[\left(\E[|X_t^{(i)}||\mathcal{F}_{t-1}]\right)^2\right]
&=\left(\E\left[\E[|X_t^{(i)}||\mathcal{F}_{t-1}]\right]\right)^2+O(n^{-2})
=\left(\E[|X_t^{(i)}|]\right)^2+O(n^{-2}),
\end{align*}
which implies
\begin{align}\label{C3}
    \E[\bX_{t}^{\top} \mathbf{P}\bX_{t}]=\sum_{i=1}^m p_i\left(\E[|X_t^{(i)}|]\right)^2 +O(n^{-2}).
\end{align}
Now, we define
\[
\eta_t:=\sum_{i=1}^m p_i\left(\E[|X_t^{(i)}|]\right)^2.
\]
Then, we would observe the drift of $\eta$. Taking expectation on both sides of \eqref{C0} by applying \eqref{C3}, we obtain
\begin{equation}\label{CC}
    \eta_{t+1}-\eta_t=-\frac{2}{n}\eta_t
    +\E\left[\frac{2}{n}\E\left[\sum_{i=1}^m p_i X_t^{(i)}\sum_{j=1}^m k_{ij}p_j\tanh(\beta X_t^{(j)})\Big|\mathcal{F}_{t-1}\right]
    \right]+O(n^{-2}).
\end{equation}
Now we estimate the second term of the right-hand side. By Cauchy--Schwarz inequality for conditional expectations, we have
\begin{align*}
    \E\left[\sum_{i=1}^m p_iX_t^{(i)}\sum_{j=1}^m k_{ij}p_j\tanh(\beta X_t^{(j)})\Big|\mathcal{F}_{t-1}\right]
    &=\sum_{i,j=1}^m k_{ij}p_ip_j \E\left[X_t^{(i)}\tanh(\beta X_t^{(j)})\Big|\mathcal{F}_{t-1}\right],
\end{align*}
\begin{equation}\label{C1}
\begin{split}
    \E\left[X_t^{(i)}\tanh(\beta X_t^{(j)})\Big|\mathcal{F}_{t-1}\right]
    &\leq\E\left[|X_t^{(i)}||\tanh(\beta X_t^{(j)})|\Big|\mathcal{F}_{t-1}\right]\\
    &\leq \left(\E\left[|X_t^{(i)}|^2\Big|\mathcal{F}_{t-1}\right]\E\left[|\tanh(\beta X_t^{(j)})|^2\Big|\mathcal{F}_{t-1}\right]\right)^{1/2}\\
    &=\E\left[|X_t^{(i)}|\Big|\mathcal{F}_{t-1}\right]\E\left[|\tanh(\beta X_t^{(j)})|\Big|\mathcal{F}_{t-1}\right]+O(n^{-1}).
\end{split}
\end{equation}
The last equality holds from the similar process of variance bound of $O(n^{-2})$.
By taking expectation on both sides of \eqref{C1} with the Cauchy--Schwarz inequality, we have
\begin{equation}\label{C2}
\begin{split}
    \E\bigg[X_t^{(i)}\tanh &(\beta X_t^{(j)})\bigg]
    \leq \E\left[\E[|X_t^{(i)}|\Big|\mathcal{F}_{t-1}]\cdot\E[|\tanh(\beta X_t^{(j)})|\Big|\mathcal{F}_{t-1}]\right]+O(n^{-1})\\
    &\leq \left(\E\left[ \left(\E[|X_t^{(i)}|\Big|\mathcal{F}_{t-1}]\right)^2 \right]\E\left[\left( \E[|\tanh(\beta X_t^{(j)})|\Big|\mathcal{F}_{t-1}] \right)^2\right]\right)^{1/2}+O(n^{-1})\\
    &= \E[|X_t^{(i)}|]\cdot\E[|\tanh(\beta X_t^{(j)})|]+O(n^{-1}),
\end{split}
\end{equation}
and the last equality follows from $O(n^{-2})$ variance bound.
We now combine \eqref{CC} and the aforementioned estimation \eqref{C2}, to obtaine
\begin{align*}
    \eta_{t+1}-\eta_t
    &\leq -\frac{2}{n}\eta_t+\frac{2}{n}\sum_{i,j=1}^m k_{ij}p_ip_j \E\left[|X_t^{(i)}|\right]\E\left[\tanh(|\beta X_t^{(j)}|)\right] +O(n^{-2})\\
    &\leq -\frac{2}{n}\eta_t+\frac{2}{n}\sum_{i,j=1}^m k_{ij}p_ip_j \E\left[|X_t^{(i)}|\right]\cdot\tanh\left(\E\left[(|\beta X_t^{(j)}|)\right]\right) +O(n^{-2}).
\end{align*}
The last inequality comes from Jensen's inequality since the $\tanh$ function is concave on the positive axis. Now, we estimate the lower bound of the right-hand side. Suppose $\eta_t>mp\epsilon^2$ for some $\epsilon>0$, where $p:=\max_{1\leq i\leq m}\{ p_i\}$, which implies $\max_{1\leq i\leq m}\E\left[|X_t^{(i)}|\right]\geq\epsilon$ (wlog, we may assume $\E\left[|X_t^{(l)}|\right]\geq\epsilon$); then, there exists a positive constant $c_\epsilon>0$ such that 
\[
\E\left[\beta|X_t^{(l)}|\right]-\tanh\left(\E\left[\beta|X_t^{(l)}|\right]\right)\geq c_\epsilon.
\]
Finally, we define a constant $C_\epsilon:=\epsilon c_\epsilon\cdot\min_{1\leq i\leq m}\{p_i^2 k_{ii}\}.$ By direct computation, we get
\begin{align*}
    \eta_t-\sum_{i,j=1}^m k_{ij}p_ip_j &\E\left[|X_t^{(i)}|\right]\cdot\tanh\left(\E\left[(|\beta X_t^{(j)}|)\right]\right)\\
    &\geq \eta_t-\sum_{i,j=1}^m k_{ij}p_ip_j \E\left[|X_t^{(i)}|\right]\cdot\E\left[(|\beta X_t^{(j)}|)\right]+\sum_{i=1}^m k_{il}p_ip_l \E\left[|X_t^{(i)}|\right]\cdot c_\epsilon\\
    &\geq \eta_t-\sum_{i,j=1}^m k_{ij}p_ip_j \E\left[|X_t^{(i)}|\right]\cdot\E\left[(|\beta X_t^{(j)}|)\right]+k_{ll}p_lp_l \E\left[|X_t^{(l)}|\right]\cdot c_\epsilon\\
    &\geq \eta_t-\sum_{i,j=1}^m \beta k_{ij}p_ip_j \E\left[|X_t^{(i)}|\right]\cdot\E\left[(| X_t^{(j)}|)\right]+C_\epsilon.
\end{align*}
We define the vector of expectation of the absolute value by
\[
\tilde{\mathbf{X}}_t^{\top}:=\left(\E[|X_t^{(i)}|]\right)_{1\leq i\leq m},\quad \tilde{\mathbf{Y}}_t:=\mathbf{V}^{\top}\mathbf{D}\tilde{\mathbf{X}}_t.
\]
Then, the right-hand side is
\begin{align*}
    \eta_t-\beta \tilde{\mathbf{X}}_t^{\top} \mathbf{D}^2\mathbf{K}\mathbf{D}^2 \tilde{\mathbf{X}}_t+C_\epsilon
    =\tilde{\mathbf{X}}_{t}^{\top}\mathbf{D}(I-\beta\mathbf{D}\mathbf{K}\mathbf{D})\mathbf{D}\tilde{\mathbf{X}}_{t}+C_\epsilon
    =\tilde{\mathbf{Y}}_{t}^{\top}(\mathbf{I}-\Lambda)\tilde{\mathbf{Y}}_{t}+C_\epsilon
    \geq C_\epsilon,
\end{align*}
since all diagonal element of $\Lambda$ is less than or equal to 1. This means that if $\eta_t>\delta$, there exist a constant $C_\delta$ such that
\begin{align*}
    \eta_{t+1}-\eta_t \leq -\frac{C_\delta}{n} +O(n^{-2}),
\end{align*}
which implies that for any constant $\delta>0$, there exists a time $t_*=O(n)$ such that $\eta_t\leq \delta$ for all $t\geq t_*$. Since $\tanh{x}\leq x-\frac{x^3}{3}+\frac{2x^5}{15}$ for $x>0$, we have $\tanh{x}\leq x-\frac{x^3}{4}$ when $x\leq 1/4$. Therefore, there exist $t_*=O(n)$ such that for $t\geq t_*,$
\newpage
\begin{align*}
    \eta_t-\sum_{i,j=1}^m k_{ij}p_ip_j &\E\left[|X_t^{(i)}|\right]\cdot\tanh\left(\E\left[|\beta X_t^{(j)}|\right]\right)\\
    &\geq \eta_t-\sum_{i,j=1}^m k_{ij}p_ip_j \E\left[|X_t^{(i)}|\right]\cdot\left(\E[|\beta X_t^{(j)}|]-\left(\E[|\beta X_t^{(j)}|]\right)^3/4\right)\\
    &\geq 
    \sum_{i,j=1}^m k_{ij}p_ip_j\E[|X_t^{(i)}|]\left(\E[|\beta X_t^{(j)}|]\right)^3/4.
\end{align*}
We define 
\[
\xi:=\beta^3\cdot\min_{1\leq i,j\leq m}\{k_{ij}p_ip_j\}/4
\]
to obtain
\begin{align*}
    \sum_{i,j=1}^m k_{ij}p_ip_j & \E[|X_t^{(i)}|]\left(\E[|\beta X_t^{(j)}|]\right)^3/4
    \geq \xi\sum_{i,j=1}^m \E[|X_t^{(i)}|]\left(\E[|X_t^{(j)}|]\right)^3\\
    &\geq \xi\sum_{i,j=1}^m\left(\E[|X_t^{(i)}|]\right)^2\left(\E[|X_t^{(j)}|]\right)^2
    = \xi \left(\sum_{i=1}^m\left(\E[|X_t^{(i)}|]\right)^2\right)^2
    \geq \frac{\xi}{p^2} \eta_t^2,
\end{align*}
that yields
\begin{align*}
    \eta_{t+1}-\eta_{t}\leq -\frac{2\xi}{n p^2}\eta_t^2+O(n^{-2}),
\end{align*}
for $t\geq t_*$. We may rewrite this formula with a sufficiently large constant $C$ and $\theta:=\frac{2\xi}{p^2}$ as
\[
    \eta_{t+1}-\eta_{t}\leq -\frac{\theta}{n}\eta_t^2+\frac{C}{n^2},
\]
which shows that for sufficiently large $n$, $\eta_t$ is decreasing for $t\geq t_*$ whereas $\eta_t\geq \sqrt{C/\theta}n^{-1/2}$. Consider the decreasing sequence $l_i:=\delta/2^i$ $(i=0,1,2,\cdots)$ and define the time 
\[
\omega_i:=\min\{t\geq t_*:\eta_t\leq l_i\}.
\]
Then, for $t\in(\omega_i,\omega_{i+1}]$, we have $l_{i+1}\leq\eta_t<l_i$, thus,
\[
\eta_{t+1}\le\eta_t-\frac{\theta l_{i+1}^2}{n}+\frac{C}{n^2}=\eta_t-\frac{\theta l_{i}^2}{4n}+\frac{C}{n^2}.
\]
Hence, we obtain
\[
\omega_{i+1}-\omega_i\leq \frac{l_i/2}{\frac{\theta l_{i}^2}{4n}-\frac{C}{n^2}}=\frac{2n}{\theta l_i}\left(1+\frac{4C}{\theta n l_i^2 -4C}\right).
\]
We define 
\[
i_*:=\min\{i:l_i\leq c_1 n^{-1/2}\},
\]
for some constant $c_1 \geq 2\sqrt{4C/\theta}$. Then, for sufficiently large $n$ and $0\leq i<i_*$,
\[
\omega_{i+1}-\omega_i\leq \frac{8n}{3\theta l_i}.
\]
By summing up from $i=0$ to $i=i_*-1$, we have
\[
\omega_{i_*}-\omega_0\leq \frac{8n}{3\theta}\sum_{i=0}^{i_*-1}\frac{1}{l_i}\leq C' \frac{n}{l_{i_*}}=O(n^{3/2}).
\]
Therefore, we conclude
\[
\omega_{i_*}\leq O(n)+O(n^{3/2})=O(n^{3/2}),
\]
which implies that it takes $O(n^{3/2})$ steps to obtain $\eta_t\leq c_1 n^{-1/2}$. In other words, there exists a constant $\tilde{c}$ such that $\eta_t\leq c_1 n^{-1/2}$ for $t\geq \tilde{c}n^{3/2}=:u_n$. Recall from \eqref{C3} that $\E[\bX_{t}^{\top} \mathbf{P}\bX_{t}]=\eta_t +O(n^{-2})$, which implies that $\E[\bX_{t}^{\top} \mathbf{P}\bX_{t}]$ has the same properties as $\eta_t$. Hence, $\bX_{t}^{\top} \mathbf{P}\bX_{t}$ is non-negative supermartingale with bounded increment. Moreover, there exists a constant $q>0$ such that $r_+(s)\wedge r_-(s) \geq q$, on the event $\{t<\tau_0\}$, there exists a constant $B$ such that
\[\E[(n^{5/4}\mathbf{X}_{t+1}^{\top}\mathbf{P}\mathbf{X}_{t+1}-n^{5/4}\mathbf{X}_t^{\top}\mathbf{P}\mathbf{X}_t)^2|\mathcal{F}_t]\geq \frac{B}{n^{3/2}} .\]
Therefore by Lemma \ref{Lem:SupMar}, if $\gamma$ is sufficiently large,
\[
\prob_\sigma(\tau_0>u_n+\gamma n^{3/2}|\mathcal{F}_{u_n})\leq \frac{c_1|n^{5/4}\mathbf{X}_{u_n}^{\top}\mathbf{P}\mathbf{X}_{u_n}|}{\sqrt{\gamma n^{3/2}}}.
\]
By taking expectation, we have
\[
\prob_\sigma(\tau_0>u_n+\gamma n^{3/2})\leq \frac{c_1 n^{3/4}}{\sqrt{\gamma} n^{3/4}}=O(\gamma^{-1/2}).
\]
to complete the proof.
\end{proof}

Note that at $t=\tau_0$, if the size of $G_i$ is even, the magnetization $S^{(i)}_t$ is zero, otherwise the magnetization $|S^{(i)}_t|=\frac{1}{n}$. Therefore, by running the modified monotone coupling after $\tau_0$, we obtain that the probability that two magnetization chains couple in additional $m$ steps, i.e., in total $\tau_0+m$ steps, is uniformly bounded away from zero. Therefore, we have $\prob_\sigma(\tau_\text{mag}\leq \bar{C} n^{3/2})\geq \bar{c}$ for some positive constants $\bar{C}$ and $\bar{c}$. By combining this with Lemma \ref{Lem:CoupleSameMag}, we conclude that there exist positive constants $C$ and $c$ such that
\[
\prob_\sigma(\tau_c\leq {C} n^{3/2})\geq {c},
\]
to conclude that $t_\text{mix}=O(n^{3/2})$.

\subsection{Lower Bound}
In this subsection, we prove the lower bound of mixing time to clarify the optimal bound.
\begin{theorem}\label{Thm:critical_lower}
When $\beta = \beta_{cr}$, then there is a constant $C_1$ such that $t_{\text{mix}}\geq C_1 n^{3/2}$.
\end{theorem}
To this end, we employ the Hubbard--Stratonovich transform to verify that the magnetization satisfies a specific limit theorem. We recall the matrices $\mathbf{\Gamma}$, $\mathbf{C}_n$, ${{\bm{U}}}(\sigma)$ from \eqref{eq:def_matrix2}. By direct calculation, we have
\begin{align*}
    H(\sigma)&
    =-\frac{\beta}{2n}\langle{{\bm{M}}}(\sigma),\mathbf{K}{{\bm{M}}}(\sigma)\rangle\\
    &=-\frac{\beta}{2n} {{\bm{M}}}^{\top}\mathbf{D}^{-1} (\mathbf{D} \mathbf{K}\mathbf{D} ) \mathbf{D}^{-1}{{\bm{M}}}
    =-\frac{1}{2n} {{\bm{M}}}^{\top}\mathbf{D}^{-1} \mathbf{V}\Lambda \mathbf{V}^{\top} \mathbf{D}^{-1}{{\bm{M}}}\\
    &=-\frac{1}{2} {{\bm{U}}}^{\top} \mathbf{C}_n {{\bm{U}}} =-\frac{1}{2}\langle \mathbf{C}_n{{\bm{U}}}(\sigma),{{\bm{U}}}(\sigma)\rangle.
\end{align*}
We claim that for $\bm{Y}_n\sim\mathcal{N}(0,\mathbf{C}_n^{-1})$, at the critical temperature, ${{\bm{U}}}_n+\bm{Y}_n$ converges to the non-trivial distribution when ${{\bm{U}}}_n$ is under the Gibbs measure $\mu_n$.
\begin{proposition}\label{Prop:Non-CLT}
Let $\bm{Y}_n\sim\mathcal{N}(0,\mathbf{C}_n^{-1})$. When $\beta=\beta_{cr}$, $\mu_n^{-1}(\bm{U}_n)+\bm{Y}_n$ converges in  distribution to a probability measure with density
\begin{align*}
    f(\bx):=\frac{1}{\tilde{Z}}\exp{\left(-\frac{1}{2}\displaystyle\sum_{i=1}^{m-1}{\Big((\lambda_i - \lambda_i^2) x_i^2\Big)}
-\frac{x_m^4}{12}\displaystyle\sum_{i=1}^{m}\frac{(V_{im})^4}{p_i}\right)}d\bx,
\end{align*}
where $\tilde{Z}$ is a normalizing constant.
\end{proposition}
Then, we can ascertain the limit of the distribution of ${{\bm{U}}}_n$ by the following lemma which holds because the weak convergence of measures is equivalent to pointwise convergence of characteristic functions.
\begin{lemma}
Suppose that for each $n$, $F_n$ and $G_n$ are independent random variables such that $F_n\rightarrow\nu$, where $\int e^{iax}d\nu(x)\neq 0$ for all $a\in\mathbb{R}$. Then, $G_n\rightarrow \mu$ if and only if $F_n+G_n\rightarrow\nu *\mu$.
\end{lemma}
\begin{remark}
Note that the last element of $Y_n$ does not contribute to the limit of ${{\bm{U}}}_n+\bm{Y}_n$. Therefore, $\Big(({U}_n)^{(j)}\Big)_{j=1,\cdots,m-1}$ converges to a normal distribution with covariance matrix $\mathbf{\Sigma}$ defined in \eqref{eq:Sigma} and the random variable $U_n^{(m)}$ converges to a distribution with Lebesgue-density $Z^{-1}\exp(-(\frac{1}{12}\sum_{i=1}^{m}(V_{im})^4/p_i)x^4)dx$.
\end{remark}
\begin{proof}[Proof of Proposition \ref{Prop:Non-CLT}]
We define $\mu_0$ to indicate the uniform measure, which means that there is no interaction (i.e. $\mathbf{J}_n\equiv 0$). Under the uniform measure,
\begin{align*}
\mathbb{E}_{\mu_0}\exp{(\langle{\bm{u}},{{\bm{M}}}\rangle)}
=\prod_{i=1}^m{\frac{1}{2^{|G_i|}}\sum_{s=-|G_i|}^{|G_i|} {\binom{|G_i|}{s}}e^{u_i s}e^{-u_i (|G_i|-s)} }
=\prod_{i=1}^m{\cosh{u_i}}^{|G_i|}.
\end{align*}
Thus, we have for any Borel set $B\in\mathcal{B}(\mathbb{R}^m)$,
\begin{align*}
    \prob({{\bm{U}}}_n+\bm{Y}_n\in B)
    &=\sum_{\sigma\in\Omega_n}\mu_n(\sigma)\mathcal{N}(0,\mathbf{C}_n^{-1})(B-\bm{u}(\sigma))\\
    &=Z_0\sum_{\sigma\in\Omega_n}\mu_n(\sigma)\int_{B}\exp\left( -\frac{1}{2}\langle \mathbf{C}_n (\bx -\bm{u}(\sigma)), \bx -\bm{u}(\sigma)\rangle  \right)d\bx \\
    &=Z_1\int_B \text{exp}\left( -\frac{1}{2}\langle \mathbf{C}_n\bx , \bx \rangle  \right)\mathbb{E}_{\mu_0}\text{exp}\left( \langle \bx ,\mathbf{C}_n\bm{u}\rangle  \right) d\bx ,
\end{align*}
where $Z_0$ and $Z_1$ are normalizing constant. Considering that the inner product $\langle \bx ,\mathbf{C}_n\bm{u}\rangle$ equals $\langle \bx ,\mathbf{C}_n\mathbf{\Gamma} \mathbf{V}^\top \mathbf{D}^{-1}\bm{m}\rangle= \langle \mathbf{D}^{-1}\mathbf{V}\mathbf{\Gamma} \mathbf{C}_n\bx ,\bm{m}\rangle$, the right-hand side equals
\begin{align*}
    Z_1\int_B \text{exp}&\left( -\frac{1}{2}\langle \mathbf{C}_n\bx , \bx \rangle  \right)\mathbb{E}_{\mu_0}\text{exp}\left( \langle \mathbf{D}^{-1}\mathbf{V} \mathbf{\Gamma} \mathbf{C}_n\bx ,\bm{m}\rangle  \right) d\bx \\
    &=Z_1\int_B \text{exp}\left( -\frac{1}{2}\langle \mathbf{C}_n\bx , \bx \rangle  \right) \cdot \text{exp}\left( \displaystyle\sum_{i=1}^{m}{np_i\cdot\log\cosh{\left((\mathbf{D}^{-1}\mathbf{V} \mathbf{\Gamma} \mathbf{C}_n\bx )_i\right)}} \right) d\bx.
\end{align*}
The right-hand side equals $Z_1\int_B \text{exp}(-\Phi_n(\bx)) d\bx$, where
\begin{align}\label{eq:Phi}
\Phi_n(\bx):= \frac{1}{2}\langle \mathbf{C}_n\bx, \bx\rangle  -\displaystyle\sum_{i=1}^{m}{np_i\cdot\log\cosh{\left((\mathbf{D}^{-1}\mathbf{V} \mathbf{\Gamma} \mathbf{C}_n\bx)_i\right)}}.
\end{align}
Now we use Taylor expansion of $\log \cosh$ as
\[
\log\cosh(x)=\frac{x^2}{2}-\frac{x^4}{12}+O(x^6),
\]
to decompose the right-hand side of \eqref{eq:Phi} by
\begin{align*}
    \bigg( \frac{1}{2}\langle \mathbf{C}_n\bx, \bx\rangle -\frac{n}{2}\displaystyle\sum_{i=1}^{m}&{p_i {\left((\mathbf{D}^{-1}\mathbf{V} \mathbf{\Gamma} \mathbf{C}_n\bx)_i\right)}^2}\bigg)\\
    &+\left(\frac{n}{12}\displaystyle\sum_{i=1}^{m}{\frac{1}{p_i}{\left((\mathbf{V} \mathbf{\Gamma} \mathbf{C}_n\bx)_i\right)}^4}\right)
    +\left(\displaystyle\sum_{i=1}^{m}{np_i \cdot O((\mathbf{D}^{-1}\mathbf{V} \mathbf{\Gamma} \mathbf{C}_n\bx)_i^6)}\right),
\end{align*}
and define each term by $\mathcal{I}_1$, $\mathcal{I}_2$, and $\mathcal{I}_3$ respectively. Here, we calculate each $\mathcal{I}_j.$\\
\noindent$\bullet$ ($\mathcal{I}_1$) : Since
\begin{align*}
    \sum_{i=1}^{m}{p_i {\left((\mathbf{D}^{-1}\mathbf{V} \mathbf{\Gamma} \mathbf{C}_n\bx)_i\right)}^2}
    =\sum_{i=1}^{m}{\left((\mathbf{V} \mathbf{\Gamma} \mathbf{C}_n\bx)_i\right)}^2
    =\sum_{i=1}^{m}{\left(( \mathbf{\Gamma} \mathbf{C}_n\bx)_i\right)}^2,
\end{align*}
we have
\begin{align*}
    \mathcal{I}_1
    =\frac{1}{2}\langle \mathbf{C}_n\bx, \bx\rangle  -\frac{n}{2}\displaystyle\sum_{i=1}^{m}{\left(( \mathbf{\Gamma} \mathbf{C}_n\bx)_i\right)}^2
    =\frac{1}{2}\displaystyle\sum_{i=1}^{m-1}{\Big((\lambda_i - \lambda_i^2) x_i^2\Big)}.
\end{align*}
\noindent$\bullet$ ($\mathcal{I}_2$) : By direct calculation, we have
\[
n^{1/4}(\mathbf{V} \mathbf{\Gamma} \mathbf{C}_n\bx)_i
= \langle \bm{V}_i, (n^{-1/4}\lambda_1x_1,\cdots,n^{-1/4}\lambda_{m-1}x_{m-1},x_m)\rangle,
\]
where $\bm{V}_i$ is the $i$th row of $\mathbf{V}$, to obtain
\begin{align*}
    \mathcal{I}_2
    &=\frac{1}{12}\displaystyle\sum_{i=1}^{m}\frac{1}{p_i}{\left(\langle \bm{V}_i, (n^{-1/4}\lambda_1x_1,\cdots,n^{-1/4}\lambda_{m-1}x_{m-1},x_m)\rangle \right)}^4\\
    &=\frac{1}{12}\displaystyle\sum_{i=1}^{m}\frac{1}{p_i}{\left(O(n^{-1/4})+V_{im}x_m)\right)}^4
    =\frac{1}{12}\displaystyle\sum_{i=1}^{m}\frac{(V_{im}x_m)^4}{p_i}+O(n^{-1/4}).
\end{align*}
\noindent$\bullet$ ($\mathcal{I}_3$) : Observe that
\begin{align*}
    \mathcal{I}_3=n\displaystyle\sum_{i=1}^{m}{\frac{1}{p_i^2}\cdot O((\mathbf{V} \mathbf{\Gamma} \mathbf{C}_n\bx)_i^6)}
.\end{align*}
Hence, we obtain
\begin{align*}
\Phi_n(\bx)
=\frac{1}{2}\displaystyle\sum_{i=1}^{m-1}{\Big((\lambda_i - \lambda_i^2) x_i^2\Big)}
+\frac{1}{12}\displaystyle\sum_{i=1}^{m}\frac{(V_{im}x_m)^4}{p_i}+O(n^{-1/4})+n\sum_{i=1}^{m}{\frac{1}{p_i^2}\cdot O((\mathbf{V} \mathbf{\Gamma} \mathbf{C}_n\bx)_i^6)}.
\end{align*}
We split $\mathbb{R}^m$ into three regions, namely the inner region $T_1=B_R(0)$, the intermediate region $T_2=F_r/B_R(0)$ and the outer region $T_3=F_r^c$, where
\[
F_r:=\left\{\bx\in\mathbb{R}^m : \Vert(n^{-1/2}x_1,\cdots,n^{-1/2}x_{m-1},n^{-1/4}x_m)\Vert_\infty\leq r\right\},
\]
for arbitrary $R>0$ and $r>0$. Now, we calculate the integration in each region. We claim that only the inner region affects the integration and the integration vanishes at the other two regions.\\
\noindent$\bullet$ In the inner region $T_1=B_R(0)$ :
Considering that the error term vanishes on compact set, we have
\begin{align*}
    \lim_{n\rightarrow\infty}\int_{B\cap T_1}\exp{\left(-\Phi_n(\bx)\right)}d\bx
    =\int_{B\cap T_1}\exp{\left(-\frac{1}{2}\displaystyle\sum_{i=1}^{m-1}{\Big((\lambda_i - \lambda_i^2) x_i^2\Big)}
-\frac{1}{12}\displaystyle\sum_{i=1}^{m}\frac{(V_{im}x_m)^4}{p_i}\right)}d\bx.
\end{align*}
\noindent$\bullet$ In the intermediate region $T_2=F_r/B_R(0)$ : We will bound the function $\Phi_n(\bx)$ in this region to yield an integration of zero. Rewrite $\Phi$ by
\[
\Phi_n(x_1,x_2,\cdots,x_m)=n\tilde{\Phi}_n(\tilde{x}_1,\cdots,\tilde{x}_m),
\]
where $\tilde{x}_i=n^{-1/2}x_i$ ($1 \leq i\leq m-1$), $\tilde{x}_m=n^{-1/4}x_m$ and
\begin{align*}
    \tilde{\Phi}_n(\bx)=\frac{1}{2}\langle \Lambda \bx,\bx\rangle-\sum_{i=1}^{m}{p_i\log\cosh\left(
    (\mathbf{D}^{-1}\mathbf{V} \Lambda \bx)_i
    \right)}.
\end{align*}
In a similar computation, we have
\begin{align*}
    \tilde{\Phi}_n(\bx)=\frac{1}{2}\sum_{i=1}^{m-1}(\lambda_i-\lambda_i^2)x_i^2+\frac{1}{12}\sum_{i=1}^m\frac{1}{p_i}\left((\mathbf{V}\Lambda \bx)_i\right)^4 -c_r \sum_{i=1}^m p_i \left((\mathbf{D}^{-1}\mathbf{V}\Lambda \bx)_i\right)^6.
\end{align*}
Note that since $T_2\in F_r$, each component of $\tilde{\bx}$ is less than or equal to $r$ and $c_r$ depends on $r$. Therefore, by taking sufficiently small $r$, we have 
\begin{align*}
    \tilde{\Phi}_n(\bx)
    \geq \frac{1}{2}\langle \tilde{C}(x_1,\cdots,x_{m-1},x_m),(x_1,\cdots,x_{m-1},x_m)\rangle,
\end{align*}
for some positive definite matrix $\tilde{C}$. This yields
\begin{align*}
    \int_{B\cap T_2}\exp{\left(-\Phi_n(\bx)\right)}d\bx
    \leq\int_{B\cap T_2}\exp{\left(- \frac{n}{2}\displaystyle\sum_{i,j=1}^{m-1}{\tilde{C}_{ij}x_ix_j} \right)}d\bx,
\end{align*}
and the right-hand side vanishes as $R\rightarrow\infty$ by dominated convergence theorem.

\noindent$\bullet$ In the outer region $T_3=F_r^c$: Because$\tilde{\Phi}_n$ has a unique minimum 0 at 0, for any $r>0$, $\inf_{\bx\in T_3}\tilde{\Phi}_n(\tilde{\bx})>0$. Therefore, we attain
\begin{align*}
    \lim_{n\rightarrow\infty}\int_{B\cap T_3}\exp{\left(-\Phi_n(\bx)\right)}d\bx
    =\lim_{n\rightarrow\infty}\int_{B\cap T_3}\exp{\left(-n\tilde{\Phi}_n(\tilde{\bx})\right)}d\bx=0,
\end{align*}
by the monotone convergence theorem.
\vspace{0.2cm}

In conclusion,
\[
\lim_{n\rightarrow\infty}\mathbb{P}({{\bm{U}}}_t+\bm{Y}_t\in B)=\frac{1}{Z}\int_{B}\exp{\left(-\frac{1}{2}\displaystyle\sum_{i=1}^{m-1}{\Big((\lambda_i - \lambda_i^2) x_i^2\Big)}
-\frac{1}{12}\displaystyle\sum_{i=1}^{m}\frac{(V_{im}x_m)^4}{p_i}\right)}d\bx,
\]
and thus
\[
\lim_{n\rightarrow\infty}\mathbb{E}\exp(i\langle \bm{u},{{\bm{U}}}_t+\bm{Y}_t\rangle )=\exp\left( -\frac{1}{2}\langle (u_1,\cdots,u_{m-1},\mathbf{\Sigma}(u_1,\cdots,u_{m-1})\rangle \right)\phi(u_m),
\]
where $\phi$ is the characteristic function of a random variable with distribution 
\[
\exp\left(-\left(\frac{1}{12}\sum_{i=1}^{m}(V_{im})^4/p_i\right)x^4\right).
\]
\end{proof}
We now prove Theorem \ref{Thm:critical_lower}.
\begin{proof}[Proof of Theorem \ref{Thm:critical_lower}]
It is sufficient to show a lower bound of the mixing time of some functions of magnetization. Let us denote the column vectors of $\mathbf{V}$ in Proposition \ref{Prop:Non-CLT} as $\vec{v}_1,\cdots,\vec{v}_m$. Note that $\vec{v}_m^\top$ is the left eigenvector of $\mathbf{Q}$ with the largest eigenvalue $1$. Then, Proposition \ref{Prop:Non-CLT} implies that each of $U_t^{(i)}$, $(1\leq i\leq m)$ defined by
\begin{align*}
    U_t^{(j)}&=n^{1/2}\sum_{i=1}^m v_{ji} S_t^{(i)}/\sqrt{p_i},\quad (j=1,2,\cdots,m-1),\\
    U_t^{(m)}&=n^{1/4}\sum_{i=1}^m v_{mi} S_t^{(i)}/\sqrt{p_i},
\end{align*}
converges to a non-trivial measure, which means that there exist a constant $A_i>0$ for $i=1,2,\cdots,m$ such that
\begin{align}\label{eq:critical_lower}
    &\Upsilon_j:=\{ \sigma : \Big|\sum_{i=1}^m v_{ji} S^{(i)}/\sqrt{p_i}\Big|\leq A_j n^{-1/2} \},\quad (j=1,2,\cdots,m-1),\nonumber\\
    &\Upsilon_m:=\{ \sigma : \Big|\sum_{i=1}^m v_{mi} S^{(i)}/\sqrt{p_i}\Big|\leq A_m n^{-1/4} \},\nonumber\\
    &\mu(\Upsilon_m)\geq\mu\left(\bigcap_{i=1}^m \Upsilon_i\right)\geq \frac{2}{3}.
\end{align}
We claim that there exists a starting state $\bS_0$ such that $\prob_{\bS_0}(\bS_t\in\Upsilon_m)\leq\frac{1}{3}$ for some suitable $t=O(n^{3/2})$. We define functions
\begin{align*}
    F_j(\bS):=\sum_{i=1}^m v_{ji} S^{(i)}/\sqrt{p_i},\quad j=1,2,\cdots,m,
\end{align*}
and choose the state $\bS_0$ satisfying $F_j(\bS_0)=2A_j n^{-1/2}$ for $j=1,2,\cdots,m-1$ and $F_m(\bS_0)=2A_m n^{-1/4}$. We may define a random variable $\tilde{\bS}_t$ that moves exactly as $\bS_t$ except when $|F_j(\bS_t)|\geq 2A_j n^{-1/2}$ $(1\leq j\leq m-1)$ or $|F_m(\bS_t)|\geq 2A_m n^{-1/4}$. On such cases, $\tilde{\bS}_t$ remains at the state with probability equal to the probability that $\bS_t$ either moves to increase the absolute value of the function $F_i({\bS}_t)$  $(1\leq i\leq m)$ or remains at the state. It follows that $\prob_{\tilde{\bS}_0}(\tilde{\bS}_t\in\Upsilon)\geq\prob_{\bS_0}({\bS}_t\in\Upsilon)$. We also define a stopping time $\tau:=\min\{t\geq 0: \tilde{\bS}_t\in \Upsilon \}$. Lastly, we note that $\vec{v}_m=\ba^{\top}=(a_1,a_2,\cdots,a_m)$ and define random variables 
\begin{align*}
U_t:=\sum_{i=1}^m a_{i}S_t^{(i)}/\sqrt{p_i},\quad
\tilde{U}_t:=\sum_{i=1}^m a_{i}\tilde{S}_t^{(i)}/\sqrt{p_i},\quad
Z_t:=\tilde{U}_0-\tilde{U}_{t\wedge \tau}.
\end{align*}
$\tilde{\bS}_t$ acts exactly as $\bS_t$ on the event $\{A_m n^{-1/4}<\tilde{U}<2A_m n^{-1/4}\}$. Thus, \eqref{eq:drift} implies
\begin{align*}
    \E_{\tilde{\bS}_0}\left[\tilde{S}_{t+1}^{(i)}|\mathcal{F}_t\right]
    =\left(1-\frac{1}{n}\right)\tilde{S}_t^{(i)}+\frac{p_i}{n}\tanh(\beta \tilde{X}_t^{(i)}) +O(n^{-2}).
\end{align*}
Therefore, we obtain
\begin{align*}
    \E_{\tilde{\bS}_0}[\tilde{U}_{t+1}|\mathcal{F}_t]
    &=\tilde{U}_t+\frac{1}{n}\sum_{i=1}^m a_i\sqrt{p_i}\left(\tanh(\beta \tilde{X}_t^{(i)})-\beta\tilde{X}_t^{(i)}\right) +O(n^{-2}),\\
    \E_{\tilde{\bS}_0}[Z_{t+1}|\mathcal{F}_t]
    &=Z_t-\frac{1}{n}\sum_{i=1}^m a_i\sqrt{p_i}\left(\tanh(\beta \tilde{X}_t^{(i)})-\beta\tilde{X}_t^{(i)}\right) +O(n^{-2}).
\end{align*}
Note that $Z_t=O(n^{-1/4})$ and $\tanh(\beta \tilde{X}_t^{(i)})-\beta\tilde{X}_t^{(i)}=O(n^{-3/4})$. In addition,
\begin{align*}
    \E_{\tilde{\bS}_0}[Z_{t+1}^2|\mathcal{F}_t]
    &\leq \left(\E_{\tilde{\bS}_0}[Z_{t+1}|\mathcal{F}_t]\right)^2+\frac{c_1}{n^2}=Z_t^2+\frac{c_2}{n^2},
\end{align*}
follows to yield $\E_{\tilde{\bS}_0}[Z_{t}^2|\mathcal{F}_t]\leq \frac{c_2 t}{n^2}$. Therefore, we obtain
\begin{align*}
    \frac{c_1 t}{n^2}\geq \E_{\tilde{\bS}_0}[Z_t^2\mathbbm{1}_{\tau>t}]
    &\geq (A_m n^{-1/4})^2\prob_{\tilde{\bS}_0}(\tau\leq t)
    =C n^{-1/2}\prob_{\tilde{\bS}_0}(\tau\leq t)\\
    &\geq C n^{-1/2}\prob_{\tilde{\bS}_0}(\tilde{\bS}_t\in\Upsilon)\geq C n^{-1/2}\prob_{{\bS}_0}({\bS}_t\in\Upsilon).
\end{align*}
Finally, we take $t= C n^{3/2}/(3c_1)$ to conclude
\[
\prob_{{\bS}_0}({\bS}_t\in\Upsilon)\leq \frac{1}{3}.
\]
Combining with \eqref{eq:critical_lower}, we conclude
\[
d\left(\frac{Cn^{3/2}}{3c_1}\right)\geq \frac{2}{3}-\frac{1}{3}=\frac{1}{3}.
\]
\end{proof}

\section{Exponential mixing time at low Temperature}\label{sec:low}
We now consider the case $\beta>\beta_{cr}$ to show that the mixing time is exponential in $n$. We claim that there are two local minima of the free energy function. For a given magnetization vector $\bm{M}=(M^{(1)},M^{(2)}\cdots,M^{(n)})$, we define 
\[
\mo_i:=M^{(i)}/(np_i),~~i=1,2,\cdots,m,\quad \vec{\mo}=(\mo_1,\mo_2,\cdots,\mo_m).
\]
Considering that $|M^{(i)}|\leq np_i$, each $\mo_i$ can take values in $[-1,1]$ for $1\leq i\leq m$. Recall from \eqref{eq:Gibbs2} that the Gibbs measure is proportional to $\exp{
\left(\frac{\beta}{2n}\sum_{1\leq i\leq j\leq m}{k_{ij}M^{(i)}M^{(j)}}\right)
}$. Therefore, by multiplying the entropy term, which is the number of possible states, we deduce that the distribution of $\vec{\mo}$ is proportional to
\[
\prod_{i=1}^m {\binom{np_i}{np_i(\mo_i+1)/2}}\prod_{1\leq i\leq j\leq m}\exp{\Big(n\beta k_{ij}p_ip_j\mo_i\mo_j\Big)}.
\]
Stirling's formula implies that this is
\begin{align*}
    \sim \prod_{1\leq i\leq j\leq m}&\exp{\Big(n\beta k_{ij}p_ip_j\mo_i\mo_j\Big)}
    \prod_{i=1}^m \exp{\left(-np_i\left(\frac{1+\mo_i}{2}\ln\frac{1+\mo_i}{2}+\frac{1-\mo_i}{2}\ln\frac{1-\mo_i}{2}\right)\right)}\\
    =&\exp\left(n\left\{\beta\sum_{1\leq i\leq j\leq m}{k_{ij}p_ip_j\mo_i\mo_j}- \sum_{i=1}^m p_i\left(\frac{1+\mo_i}{2}\ln\frac{1+\mo_i}{2}+\frac{1-\mo_i}{2}\ln\frac{1-\mo_i}{2}\right) \right\}\right),
\end{align*}
where $a_n\sim b_n$ implies $\lim_{n\rightarrow\infty}\frac{\log a_n}{\log b_n}=1$. We define a function on $\vec{\mo}$ by
\[
f(\vec{\mo}):=\beta\sum_{1\leq i\leq j\leq m}{k_{ij}p_ip_j\mo_i\mo_j}- \sum_{i=1}^m p_i\left(\frac{1+\mo_i}{2}\ln\frac{1+\mo_i}{2}+\frac{1-\mo_i}{2}\ln\frac{1-\mo_i}{2}\right),
\]
which indicates free energy. Then, we have
\begin{align*}
    \frac{\partial f(\vec{\mo})}{\partial \mo_i}=\beta\sum_{j=1}^m k_{ij}p_ip_j\mo_j-\frac{p_i}{2} \ln{\frac{1+\mo_i}{1-\mo_i}}.
\end{align*}
We seek for $\vec{\mo}$ such that $|\mo_i|\leq 1$ and $\frac{\partial f(\vec{\mo})}{\partial \mo_i}=0$ for every $1\leq i\leq m$, equivalently,
\[
\beta\sum_{j=1}^m k_{ij}p_j\mo_j = \frac{1}{2}\ln\frac{1+\mo_i}{1-\mo_i},
\]
which is known as mean-field equation. We rewrite this equation in matrix form
\[
\beta \vec{\mo}\mathbf{B}=(\tanh^{-1}\mo_1,\tanh^{-1}\mo_2,\cdots,\tanh^{-1}\mo_m).
\]
Now, we define functions by
\begin{align*}
    F(\vec{\mo}):=\beta \vec{\mo}\mathbf{B}-(\tanh^{-1}\mo_1,\cdots,\tanh^{-1}\mo_m),\quad
    \tilde{F}(\gamma):=F(\gamma\ba^{\top}).
\end{align*}
Clearly, the function $F$ is continuous on the domain $(-1,1)^m$. Note that $\ba^\top$ is the left eigenvector of the matrix $\mathbf{B}$ with eigenvalue $1/\beta_{cr}$ to obtain
\begin{align*}
    \tilde{F}(\gamma)&=\gamma (\beta/\beta_{cr}) \ba^{\top} -\left(\tanh^{-1}(\gamma a_1),\cdots,\tanh^{-1}(\gamma a_m)\right).
\end{align*}
Considering that $\beta/\beta_{cr}>1$, there exists a unique positive value $\eta<1$ such that
\begin{align*}
    \tanh^{-1}\eta =(\beta/\beta_{cr})\eta,\quad
    \begin{cases}
    \tanh^{-1}x<(\beta/\beta_{cr})x &\text{if } x<\eta,\\ \tanh^{-1}x>(\beta/\beta_{cr})x &\text{if } x>\eta,
    \end{cases}
\end{align*}
Therefore, we derive
\[
F\left(\frac{\eta}{a_{\text{max}}}\ba^{\top}\right)
=\tilde{F}\left(\frac{\eta}{a_{\text{max}}}\right)\geq\mathbf{0}.
\]
Observing that $\lim_{\vec{\mo}\rightarrow\mathbf{1}}F(\vec{\mo})=-\mathbf{\infty}$ and $F$ is a continuous function, the intermediate value theorem implies that there exists an $m$-dimensional vector $\mathbf{0}<\vec{\mo}_0<\mathbf{1}$ such that $F(\vec{\mo}_0)=\mathbf{0}.$ Clearly, we also have $F(-\vec{\mo}_0)=\mathbf{0}.$ Therefore, free energy function has two local minima. These metastable states result in exponentially slow mixing.
\newline

\noindent\thanks{\textbf{Acknowledgment.}}
{I would like to thank Professor Insuk Seo for introducing the problem and sharing his insight through numerous discussions. The author is supported by the National Research Foundation of Korea (NRF) grant funded by the Korean government (MSIT) (No. 2018R1C1B6006896).}

\bibliographystyle{amsplain}
\bibliography{Multi-component}

\end{document}